\documentclass{amsart}
\usepackage[utf8]{inputenc}
\usepackage{amsmath,amssymb,amsthm,mathtools,bm}
\usepackage{dsfont}
\usepackage{graphicx}
\usepackage{microtype}
\usepackage[hidelinks]{hyperref}

\title{Derivation of the Enskog--Enskog System from a Particle System}
\author{Zhe Chen}
\address{Laboratoire MAP5, Université Paris Cité, 45 rue des Saints-Pères, 75270 Paris Cedex 06, France}
\email{zhe.chen@u-pariscite.fr}
\newcommand{\R}{\mathbb{R}}

\newcommand{\sph}{\mathbb{S}}
\newcommand{\dd}{\mathrm{d}}
\newcommand{\dt}{\frac{\mathrm{d}}{\mathrm{d}t}}
\newcommand{\p}{\partial}
\newcommand{\onm}{\Omega^{({N_p},{N_g})}}
\newcommand{\gnm}{P^{({N_p},{N_g})}}
\renewcommand{\div}{\operatorname{div}}

\newtheorem{proposition}{Proposition}[section]
\newtheorem{theorem}{Theorem}[section]
\begin{document}
\maketitle

\begin{abstract}
We formally derive a coupled Enskog system for a gas--particle mixture from the hard-sphere BBGKY hierarchy.
The derivation retains the finite separation between colliding centres, includes collisions within and between species and explicitly states the direction of spatial shifts and the mass-dependent prefactors of the four collision operators.
The resulting system provides a kinetic starting point for two-phase flows for which the volume effect cannot be ignored.
For unit contact factors, we derive the corresponding gas and particle balance laws, and establish an entropy inequality under suitable regularity assumptions.
\end{abstract}

\section*{Introduction}

Multicomponent gas systems play a central role in various physical and engineering applications, ranging from astrophysics to industrial processes involving particulate suspensions and granular materials. Many models have been proposed to describe such systems. For binary mixtures, the most common modelling approaches include:

\begin{itemize}
\item A fluid equation coupled with another fluid equation, such as the Euler, Navier--Stokes, or Stokes equations;
\item A fluid equation coupled with a kinetic equation, such as the Vlasov or Boltzmann equation;
\item Two coupled kinetic equations.
\end{itemize}

This paper focuses on the third type of model, namely, two coupled kinetic equations. Specifically, if we denote by $f \equiv f(t,x,w)\ge 0$ and $F \equiv F(t,x,v)\ge 0$ the velocity distribution functions of masses of gas molecules and particles, respectively, at time $t$, position $x \in \mathbb{R}^3$, and velocities $w, v \in \mathbb{R}^3$, then $f$ and $F$ satisfy the following partial differential equations:
\[\left\{
\begin{aligned}
    &\p_t f+w\cdot\nabla_x f=\mathcal{R}[F,f]+\mathcal{C}[F];\\
    &\p_t F+v\cdot\nabla_x F=\mathcal{D}[f,F]+\mathcal{P}[f],
\end{aligned}\right.
\] 
where $\mathcal{R},\mathcal{C},\mathcal{D}$ and $\mathcal{P}$ stand for collision operators. For dilute binary gas mixtures, the classical Boltzmann equation offers a well-established framework for capturing the dynamics of particle interactions. 
For spatially homogeneous monatomic mixtures, \cite{GP2020} established global existence and uniqueness under suitable assumptions on the initial moments. 
For spatially inhomogeneous mixtures, \cite{HNY2007} constructed the unique global mild solution for sufficiently small, exponentially decaying initial data and a class of inverse-power interactions.
Near global equilibrium, \cite{DJMZ2016} proved spectral-gap estimates.
\cite{BD2016} subsequently established global existence, uniqueness, and exponential convergence to equilibrium for the nonlinear system. 
Complementary higher-order estimates were developed by \cite{BGPS2024}, whose micro--macro decomposition yields closed energy estimates for small perturbations of hard-sphere mixtures in one spatial dimension.
\cite{AMP2022} rigorously derived a binary Boltzmann system from hard-sphere dynamics in a mixture Boltzmann--Grad scaling.

However, as the system becomes denser, particularly with an increasing concentration of one component, the assumptions underlying the Boltzmann equation become inadequate. 
It is essential to retain the delocalization effect, which reflects the spatial correlation between particles during collisions. 
In regimes where interparticle correlations due to finite density are significant, the Enskog equation provides a more accurate description. Originally developed by D. Enskog \cite{Enskog1922} as an extension of the Boltzmann equation for dense gases, the Enskog framework accounts for the increased frequency of collisions and spatial correlations between particles. To model a multicomponent gas--particle system in such dense regimes, we employ a set of coupled Enskog equations that capture the interactions between gas molecules and dispersed particles more faithfully than dilute-limit approximations.
An important foundation is the modified Enskog theory of \cite{vBErnst2}, who developed the mixture transport description.
For the Enskog theory of multicomponent mixtures \cite{LCK1983}, its extension to polydisperse granular mixtures \cite{GDH2007}, and the H-theorem for the modified Enskog equation \cite{Resibois1978}.
Since the coupled Boltzmann equations are already hard to study, there is little literature on the rigorous theory of coupled Enskog equations. An introduction of this model can be found in Chapter 16 of \cite{Chapman1970}.

For the single-species Enskog equation,
\cite{AC1990} established global $L^1$ existence in a box together with convergence towards Boltzmann solutions as the particle diameter tends to zero. 
\cite{Lachowicz1998} justified an Enskog--Euler limit when the Knudsen number and particle-diameter scale have the same order.
\cite{EP1991} provided the existence of solutions to the revised Enskog equation.
For a spatially regularized Boltzmann--Enskog model, \cite{FRS2022} obtained Wasserstein stability estimates yielding uniqueness and continuous dependence for measure-valued solutions.

This paper is devoted to the formulation and properties of a coupled kinetic model based on the Enskog equation for a binary mixture. Our goal is to establish a framework for describing the dynamics of dense gas--particle systems, with particular attention to the effects of composition and particle interactions on evolution. 
Our derivation is grounded in the BBGKY (Bogoliubov--Born--Green--Kirkwood--Yvon) hierarchy, a fundamental tool in nonequilibrium statistical mechanics for describing the evolution of marginal distributions in a many-particle system. 
The coupled system derived here is the kinetic starting point of the formal derivation of the thick spray model in \cite{Chen2025}.

In addition to deriving the coupled kinetic equations, we also establish the macroscopic conservation laws for unit contact factors, governing mass, momentum, and energy for both species. 
These equations form the basis for continuum descriptions of multicomponent flows and are critical for applications in fluid dynamics and thermodynamics. 
Furthermore, we introduce an appropriate entropy functional for the coupled Enskog system and prove an H-theorem, demonstrating that the entropy is non-increasing in time. 
This provides a fundamental link between the microscopic reversibility of Newtonian dynamics and the macroscopic irreversibility observed in thermodynamic systems.

The outline of this paper is as follows:  
In Section \ref{sec:the classical model}, we review the classical derivation of the Boltzmann equation from a microscopic perspective, derive the conservation laws of mass, momentum, and energy for each species, and establish the entropy inequality for the coupled Boltzmann system.
Our main results are presented in Section \ref{sec:main result}. In particular, the target system is introduced in Section \ref{sec:the coupled Enskog equations}, followed by the derivation of the corresponding conservation laws in Section \ref{sec:Conservation laws for Enskog-Enskog system}, and the entropy inequality for the coupled Enskog system is established in Section \ref{sec:Entropy for Enskog-Enskog system}.
Section \ref{sec:derive coupled kinetic equations} is devoted to the formal derivation of the coupled Enskog equations from Newtonian mechanics.


\section{The classical model}\label{sec:the classical model}
\subsection{Derivation of the Boltzmann equation}

Consider a system governed by Newtonian mechanics, consisting of $N$ identical particles of a single species. Let the position and velocity of the $i$-th particle at time $t$ be denoted by $X_i(t), V_i(t) \in \mathbb{R}^3$, respectively, for $1 \le i \le N$. Each particle is assumed to be a rigid sphere of fixed radius $r > 0$. The time evolution of the system is determined by Hamilton's equations:

\begin{equation}\label{Hamilton equations single species}
\left\{
\begin{aligned}
&\frac{\mathrm{d} X_i}{\mathrm{d} t} = \frac{\partial \mathcal{H}}{\partial V_i}, \\
&\frac{\mathrm{d} V_i}{\mathrm{d} t} = -\frac{\partial \mathcal{H}}{\partial X_i}, 
\end{aligned}
\right.
\qquad \text{for } 1 \le i \le N, \text{in  } \Omega,
\end{equation}
where the Hamiltonian $\mathcal{H}$ is given by
\[
\mathcal{H}(t)=\sum_{i=1}^{N}\frac{m|V_i(t)|^2}{2},
\]
and the admissible phase space domain $\Omega$ is defined as
\[
\Omega := \left\{ (X_1,\dots,X_N) \in \mathbb{R}^{6N} : |X_i - X_j| \ge 2r \right\}.
\]
On the boundary, we impose the specular reflection condition, whose explicit formulation will be provided in Section \ref{sec:derive coupled kinetic equations}.

Grad \cite{Grad1949} derived the Boltzmann equation for the velocity distribution function $f(t,x,w)$ using the BBGKY hierarchy under the Boltzmann--Grad limit
\[
\p_t f(t,x,w)+w\cdot \nabla_x f(t,x,w)=\mathcal{B}[f](t,x,w), 
\]
where the Boltzmann collision integral is defined by
\begin{equation}\label{def:boltzmann collision integral for single species}
\mathcal{B}[f](t,x,w):=\int_{\R^3\times\sph^2}(f(t,x,w^o)f(t,x,w_1^o)-f(t,x,w)f(t,x,w_1))b(w-w_1,n)\dd  w_1\dd  n.
\end{equation}
Here $ b(z,\xi)\equiv\tilde{b}(|z|,|z\cdot \xi|)\ge 0$ denotes the collision kernel, and the pre-collision velocities $(w^o,w_1^o)$ are given by
\begin{equation}\label{collision formula between same species}\left\{
\begin{aligned}
    &w^o(w,w_1,n)=w-(w-w_1)\cdot n n\\
    &w_1^o(w,w_1,n)=w_1+(w-w_1)\cdot n n.
\end{aligned}\right.
\end{equation}


Furthermore, for a system consisting of particles with mass $m_p$ and gas molecules with mass $m_g$, let $f \equiv f(t,x,w)$ and $F \equiv F(t,x,v)$ denote the respective distribution functions of gas molecules and particles.
Using a similar technique, we can obtain the following coupled Boltzmann equations for a binary mixture:
\begin{equation}\label{coupled boltzmann equations}\left\{
    \begin{aligned}
        \p_tf+w\cdot\nabla_x f=\mathcal{B}[f]+\mathcal{B}_1[f,F]\\
        \p_tF+v\cdot\nabla_x F=\mathcal{B}_2[F,f]+\mathcal{B}[F],
    \end{aligned}\right.
\end{equation}
where the Boltzmann collision integrals $\mathcal{B}_1$ and $\mathcal{B}_2$ are defined by
\[
\begin{aligned}
&\mathcal{B}_1[f,F](w)=\int_{\R^3\times\sph^2}(f(w')F(v')-f(w)F(v))b(v-w,n)\dd  v\dd  n\\
&\mathcal{B}_2[F,f](v)=\int_{\R^3\times\sph^2}(f(w')F(v')-f(w)F(v))b(v-w,n)\dd  w\dd  n,
\end{aligned}
\]
with the pre-collision velocities given by
\begin{equation}\label{collision formula between different species}\left\{
\begin{aligned}
    &w'(v,w,n,\eta)=w+\frac{2}{1+\eta}(v-w)\cdot n n\\
    &v'(v,w,n,\eta)=v-\frac{2\eta}{1+\eta}(v-w)\cdot n n,
\end{aligned}\right.
\end{equation}
where $\eta:=m_g/m_p\le 1$ represents the mass ratio between different species.
Here $\mathcal{B}$ is the classical Boltzmann collision integral defined in \eqref{def:boltzmann collision integral for single species}.
For a detailed treatment of the two-species case in the context of coupled Boltzmann equations, we refer the reader to Chapter 8 of \cite{Chapman1970}; for a rigorous treatment, see \cite{AMP2022}.
We conclude this subsection by stating several properties of the pre-collision velocities given by \eqref{collision formula between same species} and \eqref{collision formula between different species}, whose verification is straightforward:
\begin{itemize}
    \item The pre-collision velocities are even functions with respect to $n$
    \begin{equation}\label{change the sign of n}
    \begin{aligned}
      v' (v, w, n) = v' (v, w, - n) ; \quad w' (v, w, n) = w' (v, w, - n)\\
      v^o (v, v_*, n) = v^o (v, v_*, - n) ;\quad v_*^o (v, v_*, n) = v_*^o (v, v_*, - n)
    \end{aligned}
    \end{equation}
    \item The collision process is reversible
    \begin{equation}\label{revisible collision}
    \begin{aligned}
         v' (v', w', n) = v ;\quad  w' (v', w', n) = w\\
         v^o (v^o, v_*^o, n) = v ;\quad v_*^o (v^o, v_*^o, n) = v_*
    \end{aligned}
    \end{equation}
    \item The result above shows that the absolute value of the determinant of the Jacobian of the change of variables is $1$
    \begin{equation}\label{jacobian of pre-to-post transformation}
    \left|\mathrm{det} \frac{\partial (v', w')}{\partial (v, w)} \right|= \left|\mathrm{det} \frac{\partial (v^o, v_*^o)}{\partial (v, v_*)} \right|= 1
    \end{equation}
    \item The collision is a specular reflection with respect to the normal vector $n$
    \begin{equation}\label{specular reflection}
    \begin{aligned}
      ( v' (v, w, n) - w' (v, w, n)) \cdot n = - ( v - w) \cdot n\\
       ( v^o (v, v_*, n) - v_*^o (v, v_*, n)) \cdot n = - ( v - v_*) \cdot n
    \end{aligned}
    \end{equation}
\end{itemize}
\subsection[Conservation laws for the Boltzmann--Boltzmann system]{Conservation laws for the\\ Boltzmann--Boltzmann system}

We can naturally establish conservation laws for the coupled Boltzmann equations using the following definitions:
\begin{equation}\label{def:macroscopic quantities for gas phase}
    \begin{aligned}
        &\rho_g(t,x):=\int_{\R^3}f(t,x,w)\dd  w\\
        &u_g(t,x):=\frac{\mathds{1}_{\rho_g (t,x)>0}}{\rho_g(t,x)}\int_{\R^3}wf(t,x,w)\dd  w\\
        &P_g(t,x):=\int_{\R^3}(w-u_g(t,x))^{\otimes 2}f(t,x,w)\dd  w\\
        &\theta_g(t,x):=\frac{\mathds{1}_{\rho_g (t,x)>0}}{\rho_g(t,x)}\int_{\R^3}\frac{|w-u_g(t,x)|^2}{2}f(t,x,w)\dd  w\\
        &q_g(t,x):=\frac{1}{2}\int_{\R^3}(w-u_g(t,x))|w-u_g(t,x)|^2f(t,x,w)\dd  w.
    \end{aligned}
\end{equation}
We can also define the same macroscopic quantities for liquid with subscript $\ell$.
We multiply the coupled Boltzmann equations \eqref{coupled boltzmann equations} by $1, w, \frac{1}{2}|w|^2$ for the gas molecules and $1, v, \frac{1}{2}|v|^2$ for the particles, and then integrate with respect to the velocity variables $w$ or $v$, respectively. Moreover, with the help of the following proposition, we obtain the interaction between the two species:
\begin{proposition}
For any test function $\varphi(v)\in C^\infty(\R^3)$, we have the following identity:
\begin{multline} \label{weak formula for summation of boltzmann collision integrals}
\eta \int_{\mathbb{R}^3} \mathcal{B}_1[f,F](w) \varphi(w)\, \mathrm{d}w 
+ \int_{\mathbb{R}^3} \mathcal{B}_2[F,f](v) \varphi(v)\, \mathrm{d}v \\
= \int_{\mathbb{R}^3 \times \mathbb{S}^2 \times \mathbb{R}^3} 
f(w) F(v) b(v - w, n) 
\left[\eta\varphi(w') + \varphi(v') -\eta \varphi(w) - \varphi(v) \right]
\, \mathrm{d}n \, \mathrm{d}v \, \mathrm{d}w.
\end{multline}
\end{proposition}
The proof of the proposition follows directly by performing the pre-to-post transformation on the gain terms of both collision integrals. One can also refer to \cite{BD2016,DGR2019} for details.
Because momentum and energy are conserved during collisions, as expressed by \eqref{collision conserves momentum and energy}, taking $\varphi(v)=v$ or $|v|^2$ in the weak formulation \eqref{weak formula for summation of boltzmann collision integrals} implies that the right-hand side of equation \eqref{weak formula for summation of boltzmann collision integrals} vanishes. Consequently, the force and energy are balanced between the two phases.

Hence, we obtain the equations for the continuous phase (fluid bed)
\begin{equation}\label{local conservation laws for particle of BB system}\left\{
    \begin{aligned}
        &\p_t\rho_g+\div_x(\rho_gu_g)=0\\
        &\p_t(\rho_gu_g)+\div_x(\rho_gu_g^{\otimes 2})+\div_xP_g=F^{\mathrm{int}}_B\\
        &\p_t\left(\rho_g\frac{1}{2}|u_g|^2+\rho_g\theta_g\right)+\div_x\left(\rho_g\frac{1}{2}u_g|u_g|^2+\rho_gu_g\theta_g\right)+\div_x(P_gu_g+q_g)=E^{\mathrm{int}}_B,
    \end{aligned}\right.
\end{equation}
and the equations for the dispersed phase (liquid)
\begin{equation}\label{local conservation laws for gas of BB system}\left\{
    \begin{aligned}
        &\p_t\rho_\ell +\div_x(\rho_\ell u_\ell  )=0\\
        &\p_t(\rho_\ell u_\ell  )+\div_x(\rho_\ell u_\ell  ^{\otimes 2})+\div_xP_\ell  =-\eta F^{\mathrm{int}}_B\\
        &\p_t\left(\rho_\ell \frac{1}{2}|u_\ell  |^2+\rho_\ell \theta_\ell \right)+\div_x\left(\rho_\ell \frac{1}{2}u_\ell  |u_\ell  |^2+\rho_\ell u_\ell  \theta_\ell \right)+\div_x(P_\ell  u_\ell  +q_\ell )=-\eta E^{\mathrm{int}}_B,
    \end{aligned}\right.
\end{equation}
where the interaction force $F^{\mathrm{int}}_B$ and energy exchange $E^{\mathrm{int}}_B$ are given by
\begin{equation}\label{def:boltzmann interaction momentum and energy}
    F^{\mathrm{int}}_B(t,x):=\int_{\R^3}\mathcal{B}_1[f,F](t,x,w)w\dd  w,\qquad E^{\mathrm{int}}_B(t,x):=\int_{\R^3}\mathcal{B}_1[f,F](t,x,w)\frac{|w|^2}{2}\dd  w.
\end{equation}
We define the total mass, momentum, and energy of each species by
\begin{equation}\label{def:total mass momentum and energy}
    \begin{aligned}
        &M_g(t):=\int_{\R^3}\rho_g\dd  x\\
        &Q_g(t):=\int_{\R^3}\rho_gu_g\dd  x\\
        &E_g(t):=\int_{\R^3}\rho_g\left(\frac{|u_g|^2}{2}+\theta_g\right)\dd  x\\
    \end{aligned}\qquad
    \begin{aligned}
        &M_\ell (t):=\int_{\R^3}\rho_\ell \dd  x\\
        &Q_\ell (t):=\int_{\R^3}\rho_\ell u_\ell  \dd  x\\
        &E_\ell (t):=\int_{\R^3}\rho_\ell\left( \frac{|u_\ell  |^2}{2}+ \theta_\ell\right) \dd  x.\\
    \end{aligned}
\end{equation}
From these equations, we observe that both species conserve mass individually:
\begin{equation}\label{conserve total mass}
 \dt M_g=\dt M_\ell =0.   
\end{equation}
Although neither species does not conserve momentum and energy, the total momentum and energy of the system are conserved:
\begin{equation}\label{conserve total momentum}
    \dt (\eta Q_g+Q_\ell )=0, \quad\text{and ~}\quad\dt (\eta E_g+E_\ell )=0.
\end{equation}

\subsection{Entropy for the Boltzmann--Boltzmann system}
Studying the entropy of the coupled Boltzmann equations is essential for understanding the thermodynamic behavior and convergence dynamics of multicomponent kinetic systems. 
Entropy functionals enable one to rigorously analyze the H-theorem in the coupled context, quantify dissipation rates, and derive convergence to equilibrium with explicit rates under appropriate conditions. 
Building on foundational entropy frameworks in kinetic theory (e.g. Villani's review on collisional kinetic theory \cite{VILLANI2002} and the quantitative approach to trend to equilibrium \cite{DV2005}) entropy-based techniques have also been applied in multispecies and reaction-diffusion settings (e.g., \cite{DFT2017}). Extending these methods to truly coupled Boltzmann systems allows the derivation of entropy-entropy dissipation inequalities tailored to interspecies interactions, enabling the study of exponential return to equilibrium, propagation of regularity, and robust numerical schemes that preserve structural properties, e.g. \cite{BD2016,DJMZ2016}.

The entropy of the system \eqref{coupled boltzmann equations} is simply the sum of the two Boltzmann entropies:
\begin{equation}\label{def:entropy for BB system}
    S_{\mathrm{BB}}(t):=\int_{\R^3\times \R^3}f(t,x,w)\ln f(t,x,w)\dd  x\dd  w+\int_{\R^3\times \R^3}F(t,x,v)\ln F(t,x,v)\dd  x\dd  v.
\end{equation}
We then state the following H-theorem for the binary mixture:
\begin{proposition}\label{prop:entropy inequality for BB system}
Let $f$ and $F$ be classical solutions of \eqref{coupled boltzmann equations}. Assume that $(x,v)\mapsto F(t,x,v)$ and $(x,w)\mapsto f(t,x,w)$ decay rapidly at infinity, while $(x,w)\mapsto \ln f(t,x,w)$ and $(x,v)\mapsto \ln F(t,x,v)$ have at most polynomial growth at infinity. Then
\begin{equation} 
    \dt S_{\mathrm{BB}}\le 0.
\end{equation}
\end{proposition}
Its proof is given in Appendix \ref{apd}.

\section{Main results}\label{sec:main result}

\subsection{The coupled Enskog equations}\label{sec:the coupled Enskog equations}
Our first main result is a formal derivation of the coupled Enskog equations under an Enskog closure assumption.

\begin{theorem}
\label{prop:formal derivation of coupled Enskog equations}
Consider the two-species hard-sphere system described in Section \ref{sec:derive coupled kinetic equations}, with elastic collisions and indistinguishable particles within each species.
Assume that its joint distribution satisfies the Liouville equation \eqref{liouville equation} and the collision boundary conditions \eqref{boundary condition for g}.

Let $f$ and $F$ denote the one-particle mass density distribution functions of the gas molecules and particles, respectively.
Assume that the two-particle distributions satisfy the Enskog closure relations \eqref{def:correlation functions}, including at the collision boundaries.  
Assume sufficient regularity, integrability, and decay for the integrations and changes of variables used in the derivation.

Then the first equations of the two-species BBGKY hierarchy
formally reduce to the coupled Enskog system
\begin{equation}\label{coupled Enskog equations}
\left\{
\begin{aligned}
    &\p_tf+w\cdot\nabla_{x}f=\mathcal{E}_0[f]+\mathcal{E}_1[f,F],\\
    &\p_tF+v\cdot\nabla_{x}F=\mathcal{E}_2[F,f]+\mathcal{E}_3[F],
\end{aligned}\right.
\end{equation}
where the collision operators are defined in
\eqref{def:collision integral E_0},
\eqref{def:collision integral E_1},
\eqref{def:collision integral E_2}, and
\eqref{def:collision integral E_3}.
\end{theorem}

The collision operators appearing in Theorem \ref{prop:formal derivation of coupled Enskog equations} are defined as follows:

\begin{multline}\label{def:collision integral E_0}
    \mathcal{E}_0[f]:=\frac{(2r_g)^2}{m_g}\int_{(w-w_*)\cdot n>0}\left|(w-w_*)\cdot n\right|\bigg[Y^{(0,2)}[f,f](x,x-2r_gn)\\
    \times f(x,w^o)f(x-2r_gn,w_*^o)-Y^{(0,2)}[f,f](x,x+2r_gn)f(x,w)f(x+2r_gn,w_*)\bigg]\dd  n\dd  w_*,
\end{multline}
\begin{multline}\label{def:collision integral E_1}
\mathcal{E}_1[f,F]:=\frac{(r_g+r_p)^2}{m_p}\int_{(v-w)\cdot n>0}|(v-w)\cdot n| \bigg[Y^{(1,1)}[f,F](x,x+(r_g+r_p)n)\\
\times f(x,w')F(x+(r_g+r_p)n,v')-Y^{(1,1)}[f,F](x,x-(r_g+r_p)n)f(x,w)F(x-(r_g+r_p)n,v)\bigg]\dd  n\dd  v,
\end{multline}
\begin{multline}\label{def:collision integral E_2}
    \mathcal{E}_2[F,f]:=\frac{(r_g+r_p)^2}{m_g}\int_{(v-w)\cdot n>0}|(v-w)\cdot n| \bigg[Y^{(1,1)}[F,f](x,x-(r_g+r_p)n)\\
\times F(x,v')f(x-(r_g+r_p)n,w')-Y^{(1,1)}[F,f](x,x+(r_g+r_p)n)F(x,v)f(x+(r_g+r_p)n,w)\bigg]\dd  n\dd  w,
\end{multline}
and 
\begin{multline}\label{def:collision integral E_3}
    \mathcal{E}_3[F]:=\frac{(2r_p)^2}{m_p}\int_{(v-v_*)\cdot n>0}\left|(v-v_*)\cdot n\right|\bigg[Y^{(2,0)}[F,F](x,x-2r_pn)\\
    \times F(x,v^o)F(x-2r_pn,v_*^o)-Y^{(2,0)}[F,F](x,x+2r_pn)F(x,v)F(x+2r_pn,v_*)\bigg]\dd  n\dd  v_*,
\end{multline}
where the pre-collision velocities are given by \eqref{collision formula between same species} and \eqref{collision formula between different species}, and the correlation pair functions $Y^{(i,j)}$ for $i,j\in\{0,1,2\}$ are defined later in \eqref{def:correlation functions}.

The derivation process is presented in Section \ref{sec:derive coupled kinetic equations}.
A central contribution of this work is to provide a detailed mathematical derivation of the coupled Enskog equations, with particular attention to the directions of the spatial displacements associated with collisions. 
Starting from the BBGKY hierarchy for a two-species hard-sphere system and imposing the Enskog closure assumption, we recover the coupled Enskog collision operators, including the cross-species collision terms and their mass-dependent coefficients. 
This derivation makes explicit the nonlocal structure induced by the finite separation of colliding particles and provides a self-contained basis for the subsequent analysis of conservation laws and entropy.

Please note that in this context, $f$ and $F$ represent the \emph{mass density} distribution functions of gas molecules and particles, respectively. 
If instead we wish to interpret 
\[f_{\text{num}}:=\frac{f}{m_g},\qquad \text{and}\qquad F_{\text{num}}:=\frac{F}{m_p}
\]
as \emph{number density} distribution functions, then the prefactors $1/m_g$ and $1/m_p$ in front of the corresponding collision integrals should be removed accordingly.

\subsection{Conservation laws for the Enskog--Enskog system}\label{sec:Conservation laws for Enskog-Enskog system}
The classical Boltzmann collision integral conserves mass, momentum, and energy. In other words,
\begin{equation}\label{boltzmann collision integral conserves physical quantities}
\int_{\R^3}\mathcal{B}[f](w)\begin{pmatrix}1\\w_1\\w_2\\w_3\\|w|^2\end{pmatrix}\dd  w=\begin{pmatrix}0\\0\\0\\0\\0\end{pmatrix}.
\end{equation}
Although the system governed by the Enskog--Enskog equations \eqref{coupled Enskog equations} conserves mass, total momentum, and total energy globally for unit contact factors, it does not satisfy the same local conservation laws as in \eqref{local conservation laws for particle of BB system} and \eqref{local conservation laws for gas of BB system}.
We adopt the same notation as in \eqref{def:macroscopic quantities for gas phase}, \eqref{def:boltzmann interaction momentum and energy}, and \eqref{def:total mass momentum and energy}.
From Theorem 3.3 in \cite{CCG2026}, we know that the Boltzmann--Enskog collision integral can be expressed in the following conservative form:
\begin{equation}\label{Enskog collision integral conserves physical quantities}
\mathcal{E}_0[f](w)\begin{pmatrix}1\\w_1\\w_2\\w_3\\|w|^2\end{pmatrix}=\begin{pmatrix}\div_v\mathbb{J}_0\\\div_x\mathbb{I}_1+\div_v\mathbb{J}_1\\\div_x\mathbb{I}_2+\div_v\mathbb{J}_2\\\div_x\mathbb{I}_3+\div_v\mathbb{J}_3\\\div_x\mathbb{I}_4+\div_v\mathbb{J}_4\end{pmatrix},
\end{equation}
where explicit definitions of the quantities $\mathbb{I}_k$ and $\mathbb{J}_k$ for $k=0,1,2,3,4$ are given in \cite{CCG2026}. The remaining question is how to handle the interaction terms $\mathcal{E}_1$ and $\mathcal{E}_2$. To address this question, we introduce the following proposition:


\begin{proposition}\label{prop:weak formula for enskog collision integral}
Assume that the cross-collision operators in \eqref{def:collision integral E_1} and \eqref{def:collision integral E_2} have unit contact factors and obey the collision laws \eqref{collision formula between different species}. For any smooth scalar or vector test function $\varphi(x,v)$, we have
\begin{multline}\label{weak formula for summation of enskog collision integrals}
\int_{\R^3}\mathcal{E}_1[f,F](x,w)\varphi(x,w)\dd w
+\int_{\R^3}\mathcal{E}_2[F,f](x,v)\varphi(x,v)\dd v
+\div_x I[f,F;\varphi](x)\\
=\frac{(r_g+r_p)^2}{m_g}
\int_{\R^3\times\R^3\times\sph^2}
f(x+(r_g+r_p)n,w)F(x,v)[(v-w)\cdot n]_+\\
\times\bigl[\eta\bigl(\varphi(x+(r_g+r_p)n,w')-\varphi(x+(r_g+r_p)n,w)\bigr)
+\varphi(x,v')-\varphi(x,v)\bigr]\dd n\dd v\dd w,
\end{multline}
where $[z]_+=\max\{z,0\}$ and
\begin{multline}\label{def:cross-collisional-flux}
I[f,F;\varphi](x):=\frac{(r_g+r_p)^2}{m_p}
\int_0^{r_g+r_p}\int_{\R^3\times\R^3\times\sph^2}
n[(v-w)\cdot n]_+\\
\times\bigl(\varphi(x+sn,w')-\varphi(x+sn,w)\bigr)
f(x+sn,w)\\
\times F(x-(r_g+r_p-s)n,v)\dd n\dd v\dd w\dd s.
\end{multline}
\end{proposition}
The identity follows by changing pre-collision and post-collision velocities in the gain terms and applying the fundamental theorem of calculus along the segment between the two colliding centers; see also Appendix A of \cite{CCG2026}.
When we take the test function $\varphi(x,v)=v$ or $\frac{|v|^2}{2}$ in \eqref{weak formula for summation of enskog collision integrals}, the right-hand side still vanishes by \eqref{collision conserves momentum and energy}. However, the collisional transfer flux $I[f,F;\varphi]$ remains. This term introduces corrections to the macroscopic equations, leading to modifications in the momentum and energy balances compared with the Boltzmann--Boltzmann system:

\begin{equation}\label{local conservation laws for EE system}
\left\{
\begin{aligned}
&\p_t\rho_g+\div_x(\rho_gu_g)=0,\\
&\p_t(\rho_gu_g)+\div_x(\rho_gu_g^{\otimes 2})+\div_xP_g
 =F_E^{\mathrm{int}}+\int_{\R^3}w\mathcal E_0[f]\dd w,\\
&\p_t\left(\rho_g\frac{|u_g|^2}{2}+\rho_g\theta_g\right)
 +\div_x\left(\rho_gu_g\frac{|u_g|^2}{2}+\rho_gu_g\theta_g+P_gu_g+q_g\right)\\
&\hspace{70pt}=E_E^{\mathrm{int}}+\int_{\R^3}\frac{|w|^2}{2}\mathcal E_0[f]\dd w,\\
&\p_t\rho_\ell+\div_x(\rho_\ell u_\ell)=0,\\
&\p_t(\rho_\ell u_\ell)+\div_x(\rho_\ell u_\ell^{\otimes 2})
 +\div_xP_\ell+\div_xP^{\mathrm{imb}}
 =-F_E^{\mathrm{int}}+\int_{\R^3}v\mathcal E_3[F]\dd v,\\
&\p_t\left(\rho_\ell\frac{|u_\ell|^2}{2}+\rho_\ell\theta_\ell\right)
 +\div_x\left(\rho_\ell u_\ell\frac{|u_\ell|^2}{2}
 +\rho_\ell u_\ell\theta_\ell+P_\ell u_\ell+q_\ell+Q^{\mathrm{imb}}\right)\\
&\hspace{70pt}=-E_E^{\mathrm{int}}+\int_{\R^3}\frac{|v|^2}{2}\mathcal E_3[F]\dd v.
\end{aligned}\right.
\end{equation}
Here the interaction force and energy exchange are
\begin{equation}\label{def:interaction momentum and energy}
\begin{aligned}
F_E^{\mathrm{int}}(t,x)&:=\int_{\R^3}w\mathcal E_1[f,F](t,x,w)\dd w,\\
E_E^{\mathrm{int}}(t,x)&:=\int_{\R^3}\frac{|w|^2}{2}\mathcal E_1[f,F](t,x,w)\dd w,
\end{aligned}
\end{equation}
and the cross-collisional momentum and energy fluxes are
\begin{equation}\label{def:imbalanced pressure and energy flux}
P^{\mathrm{imb}}(t,x):=I[f,F;\varphi(x,v)=v],\qquad
Q^{\mathrm{imb}}(t,x):=I\left[f,F;\varphi(x,v)=\frac{|v|^2}{2}\right].
\end{equation}
The remaining moments of $\mathcal E_0$ and $\mathcal E_3$ are the self-collisional spatial flux divergences described in \eqref{Enskog collision integral conserves physical quantities}. Thus they vanish after spatial integration, and the mass-normalized Enskog system conserves the unweighted totals
\begin{equation}\label{eq:enskog-global-conservation}
\dt(Q_g+Q_\ell)=0,\qquad \dt(E_g+E_\ell)=0.
\end{equation}
The different coefficients in \eqref{def:collision integral E_1} and \eqref{def:collision integral E_2} have already been accounted for in \eqref{weak formula for summation of enskog collision integrals}.

\paragraph{Maxwellian closure}
For unit contact factors, we consider separate Maxwellians with different densities, mean velocities and specific internal energies. The fields may depend on $t$ and $x$:
\begin{equation}\label{eq:separate-maxwellians}
\begin{aligned}
f(t,x,w)&=\rho_g(t,x)
\left(\frac{3}{4\pi\theta_g(t,x)}\right)^{3/2}
\exp\left(-\frac{3|w-u_g(t,x)|^2}{4\theta_g(t,x)}\right),\\
F(t,x,v)&=\rho_\ell(t,x)
\left(\frac{3}{4\pi\theta_\ell(t,x)}\right)^{3/2}
\exp\left(-\frac{3|v-u_\ell(t,x)|^2}{4\theta_\ell(t,x)}\right).
\end{aligned}
\end{equation}
This normalization agrees with the specific internal energies in \eqref{def:macroscopic quantities for gas phase}. It gives
\[
P_g=\frac{2}{3}\rho_g\theta_g I,\qquad
P_\ell=\frac{2}{3}\rho_\ell\theta_\ell I,\qquad q_g=q_\ell=0.
\]
The Maxwellians provide a moment closure. They are not, in general, exact solutions of the cross-collision equations. When their parameters vary in space, the self-collision Enskog operators need not vanish either, because they also sample distinct centers.

Substituting \eqref{eq:separate-maxwellians} into the moment equations \eqref{local conservation laws for EE system} gives the following system for spatially varying fields:
\begin{equation}\label{eq:maxwellian-gas-moments}
\left\{
\begin{aligned}
&\p_t\rho_g+\div_x(\rho_gu_g)=0,\\
&\p_t(\rho_gu_g)+\div_x(\rho_gu_g^{\otimes 2})
  +\frac{2}{3}\nabla_x(\rho_g\theta_g)=F_E^{\mathrm{int}}+\int_{\R^3}w\mathcal E_0[f]\dd w,\\
&\p_t\left(\rho_g\frac{|u_g|^2}{2}+\rho_g\theta_g\right)
  +\div_x\left[\rho_gu_g\left(\frac{|u_g|^2}{2}
  +\frac{5}{3}\theta_g\right)\right]\\
&\hspace{100pt}=E_E^{\mathrm{int}}
  +\int_{\R^3}\frac{|w|^2}{2}\mathcal E_0[f]\dd w,
\end{aligned}\right.
\end{equation}
and
\begin{equation}\label{eq:maxwellian-particle-moments}
\left\{
\begin{aligned}
&\p_t\rho_\ell+\div_x(\rho_\ell u_\ell)=0,\\
&\p_t(\rho_\ell u_\ell)+\div_x(\rho_\ell u_\ell^{\otimes 2})
  +\frac{2}{3}\nabla_x(\rho_\ell\theta_\ell)
  +\div_xP^{\mathrm{imb}}=-F_E^{\mathrm{int}}+\int_{\R^3}v\mathcal E_3[F]\dd v,\\
&\p_t\left(\rho_\ell\frac{|u_\ell|^2}{2}+\rho_\ell\theta_\ell\right)
  +\div_x\left[\rho_\ell u_\ell\left(\frac{|u_\ell|^2}{2}
  +\frac{5}{3}\theta_\ell\right)+Q^{\mathrm{imb}}\right]\\
&\hspace{100pt}=-E_E^{\mathrm{int}}
  +\int_{\R^3}\frac{|v|^2}{2}\mathcal E_3[F]\dd v.
\end{aligned}\right.
\end{equation}
Here all collision terms are evaluated with the Maxwellians in \eqref{eq:separate-maxwellians}. 
The self-collision moments remain nonlocal when the fields vary in space. 
Thus the Maxwellian closure retains the finite-separation transfer terms and does not impose a common pressure.

The remaining integrals cannot, in general, be replaced by purely local expressions without an additional approximation, because they couple distributions evaluated at the two collision centers.
In the Boltzmann--Grad limit, $P^{\mathrm{imb}}$, $Q^{\mathrm{imb}}$, and the self-collisional spatial fluxes vanish, while the interspecies momentum and energy exchange remain; related local collision calculations are given in the appendix of \cite{DGR2019}.
For velocity distribution functions that vary slowly over a collision diameter, one may instead expand the shifted distributions about $x$ to obtain gradient corrections, as in \cite{Chen2025}.

\subsection{Entropy for the Enskog--Enskog system}\label{sec:Entropy for Enskog-Enskog system}
The entropy of the Enskog--Enskog system is more intricate. In this section, we still restrict our attention to the Boltzmann--Enskog collision integral, which means that the correlation functions $Y^{(i,j)}$ for $i,j\in\{0,1,2\}$ are set to $1$. In addition to the kinetic part, the entropy has a potential part. With this in mind, we define the entropy of the Enskog--Enskog system as follows:
\begin{multline}\label{def:entropy for EE system}
    S_{\mathrm{EE}}(t):=\frac{1}{m_g}\int_{\R^3\times \R^3}(f\ln f)(t,x,w)\dd  x\dd  w+\frac{1}{m_p}\int_{\R^3\times \R^3}(F\ln F)(t,x,v)\dd  x\dd  v\\
    +\frac{1}{2m_p^2}\int_{\R^3\times \R^3}\beta_{2r_p}(|x-y|) \rho_\ell(t,x) \rho_\ell(t,y) \dd  x\dd  y\\
    +\frac{1}{2m_g^2}\int_{\R^3\times\R^3}\beta_{2r_g}(|x-y|) \rho_g(t,x) \rho_g(t,y) \dd  x\dd  y\\
    +\frac{1}{m_gm_p}\int_{\R^3\times \R^3}\beta_{r_p+r_g}(|x-y|) \rho_\ell(t,x) \rho_\ell(t,y) \dd  x\dd  y,
\end{multline}
where the potential function $\beta_r$ for the Enskog collision integral is defined by
\[
\beta_r(s):=\int_s^\infty \delta(\tau-r)\dd  \tau=H(r-s)=\begin{cases}1, \text{~if~} r\le s,\\0, \text{~if~} r<s.\end{cases}
\]
\begin{proposition}\label{prop:entropy inequality for EE system}
Let $f$ and $F$ be classical solutions of \eqref{coupled Enskog equations}. Assume that $(x,v)\mapsto F(t,x,v)$ and $(x,w)\mapsto f(t,x,w)$ decay rapidly at infinity, while $(x,w)\mapsto \ln f(t,x,w)$ and $(x,v)\mapsto \ln F(t,x,v)$ have at most polynomial growth at infinity. Then, for the Boltzmann--Enskog collision integrals ($Y^{(1,1)}=Y^{(0,2)}=Y^{(2,0)}\equiv 1$), we have 
\begin{equation} 
    \dt S_{\mathrm{EE}}=-\Lambda\le 0,
\end{equation}
where $\Lambda$ is defined by
\begin{multline*}
\Lambda:=-\frac{1}{2m_g^2}\int_{(\R^3)^3\times\sph^2}f(t,x+2r_gn,w_1)\lambda\left(\frac{f(t,x,w^o)f(t,x+2r_gn,w_1^o)}{f(t,x,w)f(t,x+2r_gn,w_1)}\right)\\
    \times f(t,x,w)(w-w_1)\cdot nH((w-w_1)\cdot n)\dd  x\dd  n\dd  w\dd  w_1\\
    -\frac{1}{2m_p^2}\int_{(\R^3)^3\times\sph^2}F(t,x,v)F(t,x+2r_pn,v_1)\lambda\left(\frac{F(t,x,v^o)F(t,x+2r_pn,v_1^o)}{F(t,x,v)F(t,x+2r_pn,v_1)}\right)\\
    \times (v-v_1)\cdot nH((v-v_1)\cdot n)\dd  x\dd  n\dd  v\dd  v_1\\
    -\frac{1}{m_gm_p}\int_{(\R^3)^3\times\sph^2}f(t,x+(r_g+r_p)n,w)\lambda\left(\frac{f(t,x+(r_g+r_p)n,w')F(t,x,v')}{f(t,x+(r_g+r_p)n,w)F(t,x,v)}\right)\\
    \times F(t,x,v)(v-w)\cdot nH((v-w)\cdot n)\dd  x\dd  n\dd  v\dd  w,
\end{multline*}
where $v',w',v_*^o,v^o,w^o,w_*^o$ are given by \eqref{collision formula between same species}, \eqref{collision formula between different species}, and $\lambda(z):=z-1-\ln z$.
\end{proposition}

This result is based on Theorem 3.4 in \cite{CCG2026}, which treats the entropy inequality for the single-species Boltzmann--Enskog equation. 
\begin{proof}\label{sec:proof of entropy inequality}
Similarly, we directly compute the time derivative of $S_{\mathrm{EE}}$. We decompose it into $3$ terms:
\[
\dt S_{\mathrm{EE}}=\Gamma_1+\Gamma_2+\Gamma_3,
\]
where
\begin{align*}
    \Gamma_1:=&\int_{\R^3\times\R^3}\mathcal{E}_0[f]\ln f(t,x,w)\dd  x\dd  w+\dt\left( \frac{1}{2m_g^2}\int_{\R^3}\beta_{2r_g}(|\cdot|)\star \rho_g \rho_g(t,x) \dd  x\right)\\
    \Gamma_2:=&\int_{\R^3\times\R^3}\mathcal{E}_3[F]\ln F(t,x,v)\dd  x\dd  v+\dt\left( \frac{1}{2m_p^2}\int_{\R^3}\beta_{2r_p}(|\cdot|)\star \rho_\ell \rho_\ell(t,x) \dd  x\right)\\
    \Gamma_3:=&\int_{\R^3\times \R^3}\mathcal{E}_1[f,F]\ln f\dd  x\dd  w+\int_{\R^3\times \R^3}\mathcal{E}_2[F,f]\ln F\dd  x\dd  v\\
    &\qquad +\frac{1}{m_gm_p}\dt \int_{\R^3\times \R^3}\beta_{r_g+r_p}(|x-y|)\rho_\ell(x)\rho_g(y)\dd  x\dd  y.
\end{align*}
We begin by considering $\Gamma_3$.
\begin{align*}
    &m_p\int_{\R^3\times \R^3}\mathcal{E}_1[f,F]\ln f\dd  x\dd  w+m_g\int_{\R^3\times \R^3}\mathcal{E}_2[F,f]\ln F\dd  x\dd  v\\
    =&(r_g+r_p)^2\int_{(v-w)\cdot n>0} (v-w)\cdot nf(x+(r_g+r_p)n,w)F(x,v)\\
    &\qquad\qquad\times\ln\frac{f(x+(r_g+r_p)n,w')F(x,v')}{f(x+(r_g+r_p)n,w)F(x,v)}\dd  x\dd  v\dd  w\dd  n. 
\end{align*}
By the following fundamental inequality
\begin{equation}\label{fundamental inequality}
\lambda(z):=z-1-\ln z\ge0, \text{for~}z>0 ,
\end{equation}
the quantity above is no larger than
\begin{align*}
    &(r_g+r_p)^2 \int_{(v-w)\cdot n>0}[f(x+(r_g+r_p)n,w')F(x,v')-f(x+(r_g+r_p)n,w)F(x,v)]\\
    &\qquad\qquad\qquad\qquad\qquad\qquad\qquad\qquad\qquad\qquad\qquad\qquad\times(v-w)\cdot n\dd  x\dd  v\dd  w\dd  n\\
    =&-(r_g+r_p)^2\int_{\R^3\times\R^3\times\R^3\times\sph^2}(v-w)\cdot n f(x+(r_g+r_p)n,w)F(x,v)\dd  x\dd  v\dd  w\dd  n\\
    =&\int_{(\R^3)^4}(v-w)\cdot \frac{x-y}{|x-y|} f(y,w)F(x,v)\delta(|x-y|-(r_g+r_p))\dd  x\dd  y\dd v\dd w\\
    =&\int_{(\R^3)^4}\rho_\ell(x)\rho_g(y)\frac{x-y}{|x-y|}\cdot [u_\ell(x)-u_g(y)]\delta(|x-y|-(r_g+r_p))\dd  x\dd  y\dd v\dd w.
\end{align*}
On the other hand, the equations \eqref{local conservation laws for EE system}, together with integration by parts, yield
\begin{align*}
    &\dt \int_{\R^3\times \R^3}\beta_{r_g+r_p}(|x-y|)\rho_\ell(x)\rho_g(y)\dd  x\dd  y\\
    =&-\int_{\R^3\times \R^3}\beta_{r_g+r_p}(|x-y|)\div_x(\rho_\ell(x)u_\ell(x))\rho_g(y)\dd  x\dd  y\\
    &\quad-\int_{\R^3\times \R^3}\beta_{r_g+r_p}(|x-y|)\rho_\ell(x)\div_y(\rho_g(y)u_g(y))\dd  x\dd  y\\
    =&\int_{\R^3\times \R^3}\rho_\ell(x)\rho_g(y)\{\nabla_x[\beta_{r_g+r_p}(|x-y|)]\cdot u_\ell(x)+\nabla_y[\beta_{r_g+r_p}(|x-y|)]\cdot u_g(y)\}\dd  x\dd  y\\
    =&\int_{\R^3\times \R^3}\rho_\ell(x)\rho_g(y)\frac{x-y}{|x-y|}\cdot [u_\ell(x)-u_g(y)]\beta_{r_g+r_p}'(|x-y|)\dd  x\dd  y.
\end{align*}
Recall that $\beta_s'(r)=-\delta(r-s)$. Therefore, we can rewrite the $\Gamma_3$ term as
 \begin{multline}\label{cpt:Gamma_3}
    \Gamma_3=\frac{(r_g+r_p)^2}{m_gm_p}\int_{(\R^3)^3\times\sph^2}f(t,x+(r_g+r_p)n,w)\lambda\left(\frac{f(t,x+(r_g+r_p)n,w')F(t,x,v')}{f(t,x+(r_g+r_p)n,w)F(t,x,v)}\right)\\
    \times F(t,x,v)(v-w)\cdot nH((v-w)\cdot n)\dd  x\dd  n\dd  v\dd  w\le 0.
\end{multline} 
For the same reason, or by Theorem 3.4 in \cite{CCG2026}, we have the following estimates:
\begin{multline}\label{cpt:Gamma_1}
    \Gamma_1=\frac{2r_g^2}{m_g^2}\int_{(\R^3)^3\times\sph^2}f(t,x,w)f(t,x+2r_gn,w_1)\lambda\left(\frac{f(t,x,w^o)f(t,x+2r_gn,w_1^o)}{f(t,x,w)f(t,x+2r_gn,w_1)}\right)\\
    (w-w_1)\cdot nH((w-w_1)\cdot n)\dd  x\dd  n\dd  w\dd  w_1\le0,
\end{multline}
    and
 \begin{multline}\label{cpt:Gamma_2}
    \Gamma_2=\frac{2r_p^2}{m_p^2}\int_{(\R^3)^3\times\sph^2}F(t,x,v)F(t,x+2r_pn,v_1)\lambda\left(\frac{F(t,x,v^o)F(t,x+2r_pn,v_1^o)}{F(t,x,v)F(t,x+2r_pn,v_1)}\right)\\
    (v-v_1)\cdot nH((v-v_1)\cdot n)\dd  x\dd  n\dd  v\dd  v_1\le0.
\end{multline}  
Finally, we conclude the result by combining \eqref{cpt:Gamma_1}, \eqref{cpt:Gamma_2}, and \eqref{cpt:Gamma_3}.
\end{proof}

From this proposition, we observe that the entropy of the Enskog--Enskog system is not merely the sum of the individual entropies of each species, as in the case of the coupled Boltzmann equations, but also includes additional cross-interaction terms.

\section{Formal derivation of a kinetic equation with a delocalized collision integral}\label{sec:derive coupled kinetic equations}
In this section, we aim to derive the coupled Enskog equations \eqref{coupled Enskog equations}. The technique is based on Grad's approach for deriving the Boltzmann equation for a single species. For a formal derivation, see, for instance, Chapter 2 of \cite{Cercignani1994} or Appendix A of \cite{Sone2007}. 
A rigorous short-time proof can be found in \cite{Lanford1975}, and long-time result was obtained in \cite{DHM2026}.
We begin with Newtonian mechanics. The system consists of $N_p$ “heavy” particles and $N_g$ “light” gas molecules, with positions $X_i\equiv X_i(t), Y_k\equiv Y_k(t) \in \mathbb{R}^3$ and velocities $V_i\equiv V_i(t), W_k\equiv W_k(t) \in \mathbb{R}^3$, for $1 \le i \le N_p$, $1 \le k \le N_g$. Let $r_p$ and $r_g$ denote the radii of the particles and gas molecules, respectively. For this $(N_p+N_g)$-particle system, which consists of two species, we denote the full state vector by

\[
Z(t) := (X_1,\dots,X_{N_p},Y_1,\dots,Y_{N_g},V_1,\dots,V_{N_p},W_1,\dots,W_{N_g}) \in \mathbb{R}^{6(N_p+N_g)}.
\]

The evolution of the system is governed by Hamilton's equations:

\begin{equation}\label{Hamilton equations}
\left\{
\begin{aligned}
&\frac{\mathrm{d} X_i}{\mathrm{d} t} = \frac{\partial \mathcal{H}}{\partial V_i}, \quad \frac{\mathrm{d} Y_k}{\mathrm{d} t} = \frac{\partial \mathcal{H}}{\partial W_k}, \\
&\frac{\mathrm{d} V_i}{\mathrm{d} t} = -\frac{\partial \mathcal{H}}{\partial X_i}, \quad \frac{\mathrm{d} W_k}{\mathrm{d} t} = -\frac{\partial \mathcal{H}}{\partial Y_k},
\end{aligned}
\right.
\qquad \text{for } 1 \le i \le N_p,\ 1 \le k \le N_g,  Z\in \Omega,
\end{equation}
where the Hamiltonian $\mathcal{H}$ is given by
\begin{equation*}
\mathcal{H}(t,Z(t))=\sum_{i=1}^{N_p}\frac{m_p|V_i(t)|^2}{2}+\sum_{k=1}^{N_g}\frac{m_g|W_k(t)|^2}{2}.
\end{equation*}
The admissible phase space domain $\Omega$ is defined as
\[
\Omega := \left\{ Z \in \mathbb{R}^{6(N_p+N_g)} : |X_i - X_j| \ge 2r_p,\ |X_i - Y_k| \ge r_p + r_g,\ |Y_k - Y_l| \ge 2r_g \right\}.
\]

We then define the flow map $S_t$ associated with the Hamiltonian dynamics as
\[
S_t(X,Y,V,W):=Z(t),
\]
where $Z(t)$ is given by \eqref{Hamilton equations}, and $X,V\in\R^{3N_p}, Y,W\in \R^{3N_g}$. Let ${P}={P}(t,Z(t))$ denote the joint distribution function. By Liouville's theorem, the flow $S_t$ preserves volume in phase space, meaning that the Jacobian determinant of $S_t$ is equal to $1$ for all $t$. Consequently, the distribution function satisfies the identity
\[
{P}(t,S_t(X,Y,V,W))={P}(0,X,Y,V,W).
\]
Applying the chain rule to the total time derivative, we obtain
\begin{multline*}
    0=\frac{\dd  }{\dd  t} P(t,Z(t))=\frac{\p}{\p t} P(t,Z(t))+\frac{\dd  X(t)}{\dd  t}\cdot \frac{\p}{\p X}P(t,Z(t))+ \frac{\dd  Y(t)}{\dd  t}\cdot \frac{\p}{\p Y}P(t,Z(t))\\
    +\frac{\dd  V(t)}{\dd  t}\cdot \frac{\p}{\p V}P(t,Z(t))+ \frac{\dd  W(t)}{\dd  t}\cdot \frac{\p}{\p W}P(t,Z(t)).
\end{multline*}
Combining with the Hamilton equations \eqref{Hamilton equations}, we arrive at the Liouville equation:
\begin{equation*}
    \p_t P(t,Z(t))+V \cdot \nabla_XP(t,Z(t))+ W\cdot \nabla_YP(t,Z(t))=0,\qquad \text{for  }Z\in \Omega,
\end{equation*}
which governs the evolution of the joint distribution function under the deterministic particle dynamics.

We now consider the boundary conditions imposed by elastic collisions between particles and gas molecules. Suppose the $i$-th particle collides with the $k$-th gas molecule. On the collision surface $|X_i-Y_k|=r_g+r_p$, the post-collision velocities are $V_i$ and $W_k$, while the pre-collision velocities $V_i'$ and $W_k'$ must satisfy the conservation of momentum and kinetic energy:
\begin{equation}\label{collision conserves momentum and energy}\left\{
\begin{aligned}
&m_gW_k'+m_pV_i'=m_gW_k+m_pV_i,\\
&m_g|W_k'|^2+m_p|V_i'|^2=m_g|W_k|^2+m_p|V_i|^2.
\end{aligned}\right.
\end{equation}
Additionally, since $V_i'-V_i$ and $W_k'-W_k$ are parallel to $\frac{X_i-Y_k}{|X_i-Y_k|}$, we obtain
\[
(W_k+V_i-W_k'-V_i')\cdot\frac{X_i-Y_k}{|X_i-Y_k|}=0.
\]
These constraints uniquely determine the pre-collision velocities, yielding the collision transformation on contact sphere $|X_i-Y_k|=r_g+r_p$ :
\begin{equation}\label{collision formula with mass ratio}
\begin{cases} 
  V_i'\equiv V_i'(V_i,W_k,\frac{X_i-Y_k}{|X_i-Y_k|},\eta)=V_i-\frac{2\eta}{1+\eta}(V_i-W_k)\cdot\frac{X_i-Y_k}{|X_i-Y_k|}\frac{X_i-Y_k}{|X_i-Y_k|},\\
  W_k'\equiv W_k'(V_i,W_k,\frac{X_i-Y_k}{|X_i-Y_k|},\eta)=W_k+\frac{2}{1+\eta}(V_i-W_k)\cdot \frac{X_i-Y_k}{|X_i-Y_k|}\frac{X_i-Y_k}{|X_i-Y_k|},
\end{cases}
\end{equation}
where $\eta:=\frac{m_g}{{m_p}}\le 1$ is the mass ratio. Similarly, for collisions between two particles of the same species, say particles $i$ and $j$, the pre-collision velocities $V_i^o$ and $V_j^o$ are given by
\begin{equation}\label{collision formula without mass ratio}
\begin{cases} 
  V_i^o\equiv V_i^o(V_i,V_j,\frac{X_i-X_j}{|X_i-X_j|})=V_i-(V_i-V_j)\cdot \frac{X_i-X_j}{|X_i-X_j|} \frac{X_i-X_j}{|X_i-X_j|}\delta(|X_i-X_j|-2r_p),\\
  V_j^0\equiv V_j^o(V_i,V_j,\frac{X_i-X_j}{|X_i-X_j|})=V_i+(V_i-V_j)\cdot \frac{X_i-X_j}{|X_i-X_j|} \frac{X_i-X_j}{|X_i-X_j|}\delta(|X_i-X_j|-2r_p).
\end{cases}
\end{equation}

Let $ X_{(N_p)} = (x_1, \dots, x_{N_p}) \in \mathbb{R}^{3N_p} $ and $ V_{(N_p)} = (v_1, \dots, v_{N_p}) \in \mathbb{R}^{3N_p} $ denote the position and velocity vectors of the particles. Similarly, let $ Y_{(N_g)} $ and $ W_{(N_g)} $ represent the positions and velocities of the gas molecules.

The joint velocity distribution function is defined as \[ \tilde{P}^{(N_p,N_g)} \equiv \tilde{P}^{(N_p,N_g)} (t, X_{(N_p)}, Y_{(N_g)}, V_{(N_p)}, W_{(N_g)}),\qquad \mathrm{in}~ \R_+\times\onm\times\R^{3(N_p+N_g)},\]
where the admissible domain $\onm$, which excludes overlaps between any two spheres, is defined by

\begin{multline*}
\Omega^{(N_p,N_g)}:=\{(X_{(N_p)},Y_{(N_g)}):|x_i-x_j|\ge 2r_p,|x_i-y_k|\ge r_p+r_g, |y_k-y_l|\ge 2r_g\\
\forall~i,j,k,l,~\mathrm{such~ that}~1\le i\neq j\le N_p,~1\le k\neq l\le N_g\}.
\end{multline*}
The distribution function $\tilde{P}^{(N_p,N_g)}$ evolves according to the Liouville equation:
\begin{multline}\label{liouville equation}
    \p_t\tilde{P}^{(N_p,N_g)} +\sum_{i=1}^{N_p}v_i\cdot\nabla_{x_i}\tilde{P}^{(N_p,N_g)} +\sum_{j=1}^{N_g}w_j\cdot\nabla_{y_j}\tilde{P}^{(N_p,N_g)} =0\\ \qquad\mathrm{in}~\onm\times\R^{3({N_p}+{N_g})}
\end{multline}
with the boundary conditions
\begin{equation}\label{boundary condition for g}
\left\{
\begin{array}{ll}
  \tilde{P}^{(N_p,N_g)} (t, X_{(N_p)}, Y_{(N_g)}, V_{(N_p)}, W_{(N_g)})=\tilde{P}^{(N_p,N_g)} (t, X_{(N_p)}, Y_{(N_g)}, R_{ij}V_{(N_p)}, W_{(N_g)}), &\\
  \multicolumn{2}{r}{\mathrm{on} \, \{(X_{(N_p)}, Y_{(N_g)}, V_{(N_p)}, W_{(N_g)}):(x_i-x_j)\cdot(v_i-v_j)>0, |x_i-x_j|=2r_p\}};\\
  \tilde{P}^{(N_p,N_g)} (t, X_{(N_p)}, Y_{(N_g)}, V_{(N_p)}, W_{(N_g)})=\tilde{P}^{(N_p,N_g)} (t, X_{(N_p)}, Y_{(N_g)}, C_{ij}(V_{(N_p)}, W_{(N_g)})), &\\
  \multicolumn{2}{r}{\qquad\qquad\mathrm{on} \, \{(X_{(N_p)}, Y_{(N_g)}, V_{(N_p)}, W_{(N_g)}):(x_i-y_j)\cdot(v_i-w_j)>0,|x_i-y_j|=r_p+r_g\}};\\
  \tilde{P}^{(N_p,N_g)} (t, X_{(N_p)}, Y_{(N_g)}, V_{(N_p)}, W_{(N_g)})=\tilde{P}^{(N_p,N_g)} (t, X_{(N_p)}, Y_{(N_g)}, V_{(N_p)}, R_{ij}W_{(N_g)}), &\\
  \multicolumn{2}{r}{\mathrm{on} \, \{(X_{(N_p)}, Y_{(N_g)}, V_{(N_p)}, W_{(N_g)}):(y_i-y_j)\cdot(w_i-w_j)>0,|y_i-y_j|=2r_g\}},
\end{array}
\right.
\end{equation}
where the reflection operators $R_{ij}$ and $C_{ij}$ are defined as

\[
\left\{
\begin{aligned}
R_{ij}(V_{(N_p)}) &\coloneqq \left(v_1, \dots, v_{i-1}, 
v_i^o\left(v_i, v_j, \frac{x_i - x_j}{|x_i - x_j|} \right), 
v_{i+1}, \dots, v_{j-1}, \right. \\
&\quad \left. v_j^o\left(v_i, v_j, \frac{x_i - x_j}{|x_i - x_j|} \right), 
v_{j+1}, \dots, v_{N_p} \right), 
\quad 1 \le i < j \le N_p, \\[1ex]
C_{ij}(V_{(N_p)}, W_{(N_g)}) &\coloneqq \left(v_1, \dots, 
v_i'\left(v_i, w_j, \frac{x_i - y_j}{|x_i - y_j|}, \eta \right), 
\dots, v_{N_p}, w_1, \dots ,\right. \\
&\quad \left. w_j'\left(v_i, w_j, \frac{x_i - y_j}{|x_i - y_j|}, \eta \right), 
\dots, w_{N_g} \right), \quad 1 \le i \le N_p, \; 1 \le j \le N_g, \\[1ex]
R_{ij}(W_{(N_g)}) &\coloneqq \left(w_1, \dots, w_{i-1}, 
w_i^o\left(w_i, w_j, \frac{y_i - y_j}{|y_i - y_j|} \right), 
w_{i+1}, \dots, w_{j-1} ,\right. \\
&\quad \left. w_j^o\left(w_i, w_j, \frac{y_i - y_j}{|y_i - y_j|} \right), 
w_{j+1}, \dots, w_{N_g} \right), 
\quad 1 \le i < j \le N_g,
\end{aligned}
\right.
\]
where the pre-collision velocities are given by \eqref{collision formula between same species} and \eqref{collision formula between different species}.

We now clarify this boundary condition. Consider a collision between particle $i$ and gas molecule $k$. The pre-collision configuration corresponds to the particles approaching each other, which is characterized by the condition:
\[
(x_i - y_k) \cdot (v_i' - w_k') < 0.
\]
Conversely, after the collision, the relative velocity should be directed outward, indicating that the particles are moving apart:
\[
(x_i - y_k) \cdot (v_i - w_k) > 0.
\]
Assuming sufficient regularity of the distribution function, it is reasonable to impose the boundary condition on the interaction sphere $\{|x_i - y_k| = r_p + r_g\}$ in the following form:
\begin{multline*}
    P(t, X_{(N_p)}, Y_{(N_g)}, V_{(N_p)}, W_{(N_g)}) H\big(- (x_i - y_k) \cdot (v_i' - w_k')\big)  \\
    = P(t, X_{(N_p)}, Y_{(N_g)}, C_{ik}(V_{(N_p)}, W_{(N_g)})) H\big((x_i - y_k) \cdot (v_i - w_k)\big) .
\end{multline*}

By the collision property \eqref{specular reflection}, this boundary condition is naturally imposed on the outgoing hemisphere:
\[
\left\{ |x_i - y_k| = r_p + r_g,\quad (x_i - y_k) \cdot (v_i - w_k) > 0 \right\}.
\]
See Figure \ref{fig:collision between different species} for an illustration:

\begin{figure}[htbp]
  \centering
  \includegraphics[width=0.8\textwidth]{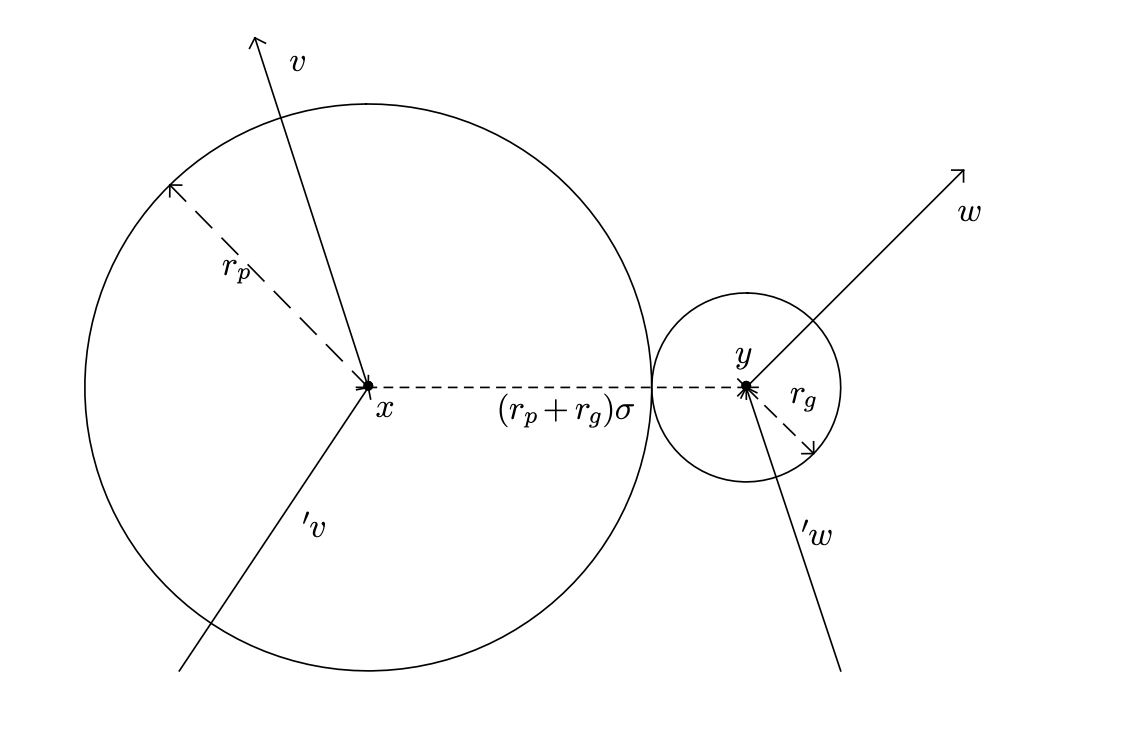}
  \caption{Collision Between Two Dissimilar Particles}
  \label{fig:collision between different species}
\end{figure}

We symmetrize the distribution function $\tilde{P}^{(N_p,N_g)}$ by
\begin{multline}\label{definition of g bar}
    \bar{P}^{(N_p,N_g)}:=\frac{1}{N_p!N_g!}\sum_{\sigma,\tau}\tilde{P}^{(N_p,N_g)}(t,x_{\sigma(1)},\dots,x_{\sigma(N_p)},y_{\tau(1)},\dots,y_{\tau(N_g)},\\
    v_{\sigma(1)},\dots,v_{\sigma(N_p)},w_{\tau(1)},\dots,w_{\tau(N_g)}),
\end{multline}
where $\sigma$ and $\tau$ are permutations in $S_{N_p}$ and $S_{N_g}$, respectively. Hence, the function $\bar{P}^{(N_p,N_g)}$ is permutation invariant, i.e.,
\begin{multline}\label{symmetry of g}
\bar{P}^{(N_p, N_g)}(t, X_{(N_p)}, Y_{(N_g)}, V_{(N_p)}, W_{(N_g)})= 
\bar{P}^{(N_p,N_g)}(t,x_{\sigma(1)},\dots,x_{\sigma(N_p)}\\
,y_{\tau(1)},\dots,y_{\tau(N_g)},v_{\sigma(1)},\dots,v_{\sigma(N_p)},w_{\tau(1)},\dots,w_{\tau(N_g)}),
\end{multline}  
ensuring indistinguishability of molecules within the same species.
Next, we extend $\bar{P}^{(N_p,N_g)}$ to $\R_+\times\R^{6(N_p+N_g)}$ by
\begin{multline}
\gnm(t,X_{(N_p)},Y_{(N_g)}):=\bar{P}^{(N_p,N_g)}(t,X_{(N_p)},Y_{(N_g)})\mathds{1}_{\onm}(X_{(N_p)},Y_{(N_g)})\\=\begin{cases}
  \bar{P}^{(N_p,N_g)},\quad \mathrm{if}~(X_{(N_p)},Y_{(N_g)})\in \onm\\
  0,\qquad\qquad \mathrm{if}~(X_{(N_p)},Y_{(N_g)})\notin \onm
\end{cases}.
\end{multline}
We remake here that $\bar{P}$ still satisfies the boundary condition \eqref{boundary condition for g}.
We denote the variables by $X_{(s)} = (x_1, \dots, x_s) \in \R^{3s}$, $Y_{(l)} = (y_1, \dots, y_l) \in \R^{3l}$, $V_{(s)} = (v_1, \dots, v_s) \in \R^{3s}$, and $W_{(l)} = (w_1, \dots, w_l) \in \R^{3l}$. Similarly, the complementary variables are defined as $ X_{(N_p)}^{(s)} := (x_{s+1}, \dots, x_{N_p}) \in \R^{3(N_p - s)}$, $ Y_{(N_g)}^{(l)} := (y_{l+1}, \dots, y_{N_g}) \in \R^{3(N_g - l)}$, $ V_{(N_p)}^{(s)} := (v_{s+1}, \dots, v_{N_p}) \in \R^{3(N_p - s)}$, and $ W_{(N_g)}^{(l)} := (w_{l+1}, \dots, w_{N_g}) \in \R^{3(N_g - l)}$.
We then introduce the marginals of $\gnm$:
\begin{equation}\label{definition of marginal}
    P^{(s,l)}(t,X_{(s)},Y_{(l)},V_{(s)},W_{(l)}):=\int_{\R^{6(N_p+N_g-s-l)}}\gnm\dd  X^{(s)}_{(N_p)}\dd  Y^{(l)}_{(N_g)}\dd  V^{(s)}_{(N_p)}\dd  W^{(l)}_{(N_g)}.
\end{equation}

To study the dynamics of the first marginal for both species, we integrate \eqref{liouville equation} over the admissible domain for a single molecule:
\begin{align*}
\onm_{(1,0)}:=\{&(x_2,\dots,x_{N_p},Y_{(N_g)})\in\R^{3({N_p}+{N_g}-1)}:|x_i-x_1|\ge 2r_p,|x_1-y_j|\ge r_p+r_g,\\
&|x_i-x_j|\ge2r_p,|y_j-y_l|\ge 2r_g,|x_i-y_j|\ge r_g+r_p\\
&\forall 2\le i\neq j\le {N_p},1\le k\neq l\le {N_g}\}\\
\onm_{(0,1)}:=\{&(X_{(N_p)},y_2,\dots,y_{(N_p)})\in\R^{3({N_p}+{N_g}-1)}:|x_i-x_k|\ge 2r_p,|x_i-y_1|\ge r_p+r_g,\\
&|x_i-x_j|\ge2r_p,|y_j-y_l|\ge 2r_g,|x_i-y_j|\ge r_g+r_p\\
&\forall 2\le i\neq j\le {N_p},1\le k\neq l\le {N_g}\}.
\end{align*}
The equation for the particle marginal is obtained from \eqref{liouville equation} as follows:
\begin{multline}\label{cpt:integrate the first marginal for particles}
\int_{\onm_{(1,0)}\times\R^{3(N_p+N_g-1)}}\bigg(\p_tP^{(N_p,N_g)} +\sum_{i=1}^{N_p}v_i\cdot\nabla_{x_i}P^{(N_p,N_g)}\\ +\sum_{j=1}^{N_g}w_j\cdot\nabla_{y_j}P^{(N_p,N_g)} \bigg)\dd  X^{(2)}_{(N_p)}\dd  Y_{(N_g)}\dd  V^{(2)}_{(N_p)}\dd  W_{(N_g)}=0.
\end{multline}
If we assume $\p_t\gnm\in L^1(\R_+\times\R^{6(N_p+N_g)})$, the definition of the marginals \eqref{definition of marginal} gives
\begin{equation}\label{cpt:p_tg}
\int_{\onm_{(1,0)}\times\R^{3(N_p+N_g-1)}}\p_tP^{(N_p,N_g)}\dd  X^{(2)}_{(N_p)}\dd  Y_{(N_g)}\dd  V^{(2)}_{(N_p)}\dd  W_{(N_g)}=\p_tP^{(1,0)}.
\end{equation}
If $v_1\cdot\nabla_{x_1}\gnm\in L^1(\R_+\times\R^{6(N_p+N_g)})$, we change the order of integration. Because of the boundary, this gives
\begin{multline*}
\int_{\onm_{(1,0)}\times\R^{3(N_p+N_g-1)}}v_1\cdot\nabla_{x_1}\gnm\dd  X^{(2)}_{(N_p)}\dd  Y_{(N_g)}\dd  V^{(2)}_{(N_p)}\dd  W_{(N_g)}=v_1\cdot\nabla_{x_1}P^{(1,0)}\\
-\sum_{j=2}^{N_p}\int_{\onm_{(1,0)}\times\R^{3(N_p+N_g-1)}}v_1\cdot\frac{x_1-x_j}{|x_1-x_j|}\delta(|x_1-x_j|-2r_p)\gnm\dd  X^{(2)}_{(N_p)}\dd  Y_{(N_g)}\dd  V^{(2)}_{(N_p)}\dd  W_{(N_g)}\\
-\sum_{k=1}^{N_g}\int_{\onm_{(1,0)}\times\R^{3(N_p+N_g-1)}}v_1\cdot\frac{x_1-y_k}{|x_1-y_k|}\delta(|x_1-y_k|-r_g-r_p)\gnm\dd  X^{(2)}_{(N_p)}\dd  Y_{(N_g)}\dd  V^{(2)}_{(N_p)}\dd  W_{(N_g)}.
\end{multline*}
Therefore, the symmetry property \eqref{symmetry of g} allows us to simplify the integral by replacing all indices $ 2 \le j \le N_p $ with $ 2 $, and $ 1 \le k \le N_g $ with $ 1 $, because of the assumed symmetry within each species.
\begin{multline}\label{cpt:v_1cdotnabla_x_1g}
\int_{\onm_{(1,0)}\times\R^{3(N_p+N_g-1)}}v_1\cdot\nabla_{x_1}\gnm\dd  X^{(2)}_{(N_p)}\dd  Y_{(N_g)}\dd  V^{(2)}_{(N_p)}\dd  W_{(N_g)}\\
=v_1\cdot\nabla_{x_1}P^{(1,0)}
-(N_p-1)\int_{\R^{6}}v_1\cdot\frac{x_1-x_2}{|x_1-x_2|}\delta(|x_1-x_2|-2r_p)P^{(2,0)}\dd  x_2\dd  v_2\\
-N_g\int_{\R^{6}}v_1\cdot\frac{x_1-y_1}{|x_1-y_1|}\delta(|x_1-y_1|-r_g-r_p)P^{(1,1)}\dd  y_1\dd  w_1.
\end{multline}

For $2\le i \le N_p$, we apply Gauss's theorem to the boundary, which consists of three parts:
\begin{itemize}
    \item A sphere centered at the unintegrated variable $x_1$: $\{(X_{(N_p)},Y_{(N_p)}):|x_i-x_1|=2r_p\}$;
    \item Spheres centered at the integrated variables $x_j$ for $2\le j\neq i\le N_p$: $\{(X_{(N_p)},Y_{(N_p)}):|x_i-x_1|=2r_p\}$;
    \item Spheres centered at the other species $y_k$ for $1\le k\le N_g$: $\{(X_{(N_p)},Y_{(N_p)}):|x_i-y_k|=r_g+r_p\}$.
\end{itemize}
\begin{multline*}
\int_{\onm_{(1,0)}\times\R^{3(N_p+N_g-1)}}v_i\cdot\nabla_{x_i}P^{(N_p,N_g)}\dd  X^{(2)}_{(N_p)}\dd  Y_{(N_g)}\dd  V^{(2)}_{(N_p)}\dd  W_{(N_g)}\\
=-\int_{\onm_{(1,0)}\times\R^{3(N_p+N_g-1)}}v_i\cdot\frac{x_i-x_1}{|x_i-x_1|}P^{(N_p,N_g)}\delta(|x_i-x_1|-2r_p)\dd  X^{(2)}_{(N_p)}\dd  Y_{(N_g)}\dd  V^{(2)}_{(N_p)}\dd  W_{(N_g)}\\
-\sum_{2\le j\neq i\le N_p}\int_{\onm_{(1,0)}\times\R^{3(N_p+N_g-1)}}v_i\cdot\frac{x_i-x_j}{|x_i-x_j|}P^{(N_p,N_g)}\delta(|x_i-x_j|-2r_p)\dd  X^{(2)}_{(N_p)}\dd  Y_{(N_g)}\dd  V^{(2)}_{(N_p)}\dd  W_{(N_g)}\\
-\sum_{1\le k\le N_g}\int_{\onm_{(1,0)}\times\R^{3(N_p+N_g-1)}}v_i\cdot\frac{x_i-y_k}{|x_i-y_k|}P^{(N_p,N_g)}\delta(|x_i-y_k|-r_g-r_p)\dd  X^{(2)}_{(N_p)}\dd  Y_{(N_g)}\dd  V^{(2)}_{(N_p)}\dd  W_{(N_g)}.
\end{multline*}
Thanks to the symmetry of the distribution function \eqref{symmetry of g}, we can, without loss of generality, relabel the indices in the integrals and summations: replace $i$ by $2$ in the first integral, replace each pair $ (i,j) $ with $ (2,3) $ for $ 2 \le i \neq j \le N_p $ in the first summation, and replace each pair $ (i,k) $ with $ (2,1) $ in the second summation. This leads to the following simplified expression:
\begin{multline*}
\int_{\onm_{(1,0)}\times\R^{3(N_p+N_g-1)}}v_i\cdot\nabla_{x_i}P^{(N_p,N_g)}\dd  X^{(2)}_{(N_p)}\dd  Y_{(N_g)}\dd  V^{(2)}_{(N_p)}\dd  W_{(N_g)}\\
   = -\int_{\R^6}v_2\cdot\frac{x_2-x_1}{|x_2-x_1|}P^{(2,0)}(t,x_1,x_2,v_1,v_2)\delta(|x_1-x_2|-2r_p)\dd  x_2\dd  v_2\\
    -(N_p-2)\int_{\R^{12}}v_2\cdot\frac{x_2-x_3}{|x_2-x_3|}P^{(3,0)}(t,x_1,x_2,x_3,v_1,v_2,v_3)\delta(|x_2-x_3|-2r_p)\dd  x_2\dd  x_3\dd  v_2\dd  v_3\\
    -N_g\int_{\R^{12}}v_2\cdot\frac{x_2-y_1}{|x_2-y_1|}P^{(2,1)}(t,x_1,x_2,y_1,v_1,v_2,w_1)\delta(|x_2-y_1|-r_g-r_p)\dd  y_1\dd  x_2\dd  w_1 \dd  v_2.
\end{multline*}
We then sum the above formula over all $2\le i\le N_p$, which yields
\begin{multline}\label{cpt:v_icdotnabla_x_ig}
  \sum_{i=2}^{N_p}\int_{\onm_{(1,0)}\times\R^{3(N_p+N_g-1)}}v_i\cdot\nabla_{x_i}P^{(N_p,N_g)}\dd  X^{(2)}_{(N_p)}\dd  Y_{(N_g)}\dd  V^{(2)}_{(N_p)}\dd  W_{(N_g)}\\
    =(N_p-1)\int_{\R^6}v_2\cdot\frac{x_2-x_1}{|x_2-x_1|}P^{(2,0)}\dd  x_2\dd  v_2\\
    +(N_p-1)(N_p-2)\int_{\R^{12}}v_2\cdot\frac{x_2-x_3}{|x_2-x_3|}P^{(3,0)}\delta(|x_2-x_3|-2r_p)\dd  x_2\dd  x_3\dd  v_2\dd  v_3\\
    +(N_p-1)N_g\int_{\R^{12}}v_2\cdot\frac{x_2-y_1}{|x_2-y_1|}P^{(2,1)}\delta(|x_2-y_1|-r_g-r_p)\dd  y_1\dd  x_2\dd  w_1 \dd  v_2.
\end{multline}
For $1\le k\le N_g$, we compute the following integral using Gauss's theorem on the boundary 
\[
\bigcup_{1\le i\le N_p}\{y_k\in\R^3:|y_k-x_i|=r_g+r_p \}\bigcup_{1\le l\neq k\le N_g} \{y_k\in\R^3:|y_k-y_l|=2r_g\},
\] then we have
\begin{multline*}
    \int_{\onm_{(1,0)}\times\R^{3(N_p+N_g-1)}}w_k\cdot\nabla_{y_k}P^{(N_p,N_g)}\dd  X^{(2)}_{(N_p)}\dd  Y_{(N_g)}\dd  V^{(2)}_{(N_p)}\dd  W_{(N_g)}\\
=\int_{\onm_{(1,0)}\times\R^{3(N_p+N_g-1)}}w_k\cdot\frac{y_k-x_1}{|y_k-x_1|}P^{(N_p,N_g)}\delta(|y_k-x_1|-r_g-r_p)\\
\times\dd  X^{(2)}_{(N_p)}\dd  Y_{(N_g)}\dd  V^{(2)}_{(N_p)}\dd  W_{(N_g)}\\
+\sum_{2\le i\le N_p}\int_{\onm_{(1,0)}\times\R^{3(N_p+N_g-1)}}w_k\cdot\frac{y_k-x_i}{|y_k-x_i|}P^{(N_p,N_g)}\delta(|y_k-x_i|-r_g-r_p)\\
\times\dd  X^{(2)}_{(N_p)}\dd  Y_{(N_g)}\dd  V^{(2)}_{(N_p)}\dd  W_{(N_g)}\\
+\sum_{1\le l\neq k\le N_g}\int_{\onm_{(1,0)}\times\R^{3(N_p+N_g-1)}}w_k\cdot\frac{y_k-y_l}{|y_k-y_l|}P^{(N_p,N_g)}\delta(|y_k-y_l|-2r_g)\\
\times\dd  X^{(2)}_{(N_p)}\dd  Y_{(N_g)}\dd  V^{(2)}_{(N_p)}\dd  W_{(N_g)}.
\end{multline*}
The symmetry of $\bar{P}^{(N_p,N_g)}$ and the definition of the marginals give:
\begin{multline}\label{cpt:w_kcdotnabla_y_kg}
    \int_{\onm_{(1,0)}\times\R^{3(N_p+N_g-1)}}w_k\cdot\nabla_{y_k}P^{(N_p,N_g)}\dd  X^{(2)}_{(N_p)}\dd  Y_{(N_g)}\dd  V^{(2)}_{(N_p)}\dd  W_{(N_g)}\\
=-\int_{\R^6}w_1\cdot\frac{y_1-x_1}{|y_1-x_1|}P^{(1,1)}\delta(|y_1-x_1|-r_g-r_p)\dd  y_1\dd  w_1\\
-(N_p-1)\int_{\R^{12}}w_1\cdot\frac{y_1-x_2}{|y_1-x_2|}P^{(2,1)}\delta(|y_1-x_2|-r_g-r_p)\dd  y_1\dd  x_2\dd  w_1\dd  v_2\\
-(N_g-1)\int_{\R^{12}}w_1\cdot\frac{y_1-y_2}{|y_1-y_2|}P^{(1,2)}\delta(|y_1-y_2|-2r_g)\dd  y_1\dd  y_2\dd  w_1\dd  w_2.
\end{multline}

Substituting \eqref{cpt:p_tg}, \eqref{cpt:v_1cdotnabla_x_1g}, \eqref{cpt:v_icdotnabla_x_ig}, and summing \eqref{cpt:w_kcdotnabla_y_kg} over $k$ into \eqref{cpt:integrate the first marginal for particles}, we obtain:
\begin{multline}\label{cpt:before using boundary conditions}
    \p_tP^{(1,0)}+v_1\cdot\nabla_{x_1}P^{(1,0)}\\
    =(N_p-1)(N_p-2)\int_{\R^{12}}v_2\cdot\frac{x_2-x_3}{|x_2-x_3|}P^{(3,0)}\delta(|x_2-x_3|-2r_p)\dd  x_2\dd  x_3\dd  v_2\dd  v_3\\
+N_g(N_g-1)\int_{\R^{12}}w_1\cdot\frac{y_1-y_2}{|y_1-y_2|}P^{(1,2)}\delta(|y_1-y_2|-2r_g)\dd  y_1\dd  y_2\dd  w_1\dd  w_2\\
-(N_p-1)N_g\int_{\R^{12}}(v_2-w_1)\cdot\frac{y_1-x_2}{|y_1-x_2|}P^{(2,1)}\delta(|y_1-x_2|-r_g-r_p)\dd  y_1\dd  x_2\dd  w_1\dd  v_2\\
+(N_p-1)\int_{\R^{6}}(v_1-v_2)\cdot\frac{x_1-x_2}{|x_1-x_2|}\delta(|x_1-x_2|-2r_p)P^{(2,0)}\dd  x_2\dd  v_2\\
+N_g\int_{\R^{6}}(v_1-w_1)\cdot\frac{x_1-y_1}{|x_1-y_1|}\delta(|x_1-y_1|-r_g-r_p)P^{(1,1)}\dd  y_1\dd  w_1.
\end{multline}

Next, we simplify the expression term by term. We will show that the interaction involving three molecules ultimately contributes zero. Meanwhile, the pairwise collisions between two molecules can be reformulated as collision integrals, specifically, as either Boltzmann or Enskog collision integrals. We take the first integral in \eqref{cpt:before using boundary conditions} as an example. By the symmetry of $\bar{P}^{(N_p,N_g)}$, we can exchange the indices $2$ and $3$ and hence obtain
\begin{multline*}
    \int_{\R^{12}}v_2\cdot\frac{x_2-x_3}{|x_2-x_3|}P^{(3,0)}\delta(|x_2-x_3|-2r_p)\dd  x_2\dd  x_3\dd  v_2\dd  v_3\\
    =\int_{\R^{12}}v_3\cdot\frac{x_3-x_2}{|x_2-x_3|}P^{(3,0)}\delta(|x_2-x_3|-2r_p)\dd  x_2\dd  x_3\dd  v_2\dd  v_3\\
    =\frac{1}{2}\int_{\R^{12}}(v_2-v_3)\cdot\frac{x_2-x_3}{|x_2-x_3|}P^{(3,0)}\delta(|x_2-x_3|-2r_p)\dd  x_2\dd  x_3\dd  v_2\dd  v_3.
\end{multline*}
Next, we decompose  $(v_2 - v_3) \cdot \frac{x_2 - x_3}{|x_2 - x_3|}$ into its positive and negative parts:
\begin{multline*}
(v_2 - v_3) \cdot \frac{x_2 - x_3}{|x_2 - x_3|} = \left|(v_2 - v_3) \cdot \frac{x_2 - x_3}{|x_2 - x_3|}\right|H\left[(v_2 - v_3) \cdot (x_2 - x_3)\right] \\
- \left|(v_2 - v_3) \cdot \frac{x_2 - x_3}{|x_2 - x_3|}\right|H\left[-(v_2 - v_3) \cdot (x_2 - x_3)\right].
\end{multline*}
For the negative part, we perform the change of variables $V_{(N_p)}\mapsto R_{23}(V_{(N_p)})$ with the Jacobian given by \eqref{jacobian of pre-to-post transformation}. The boundary condition \eqref{boundary condition for g} guarantees the invariance of $P^{(3,0)}$ under the transformation. Hence, we obtain:
\begin{multline*}
\int_{\R^{12}}(v_2-v_3)\cdot\frac{x_2-x_3}{|x_2-x_3|}P^{(3,0)}\delta(|x_2-x_3|-2r_p)\dd  x_2\dd  x_3\dd  v_2\dd  v_3\\
=\int_{\R^{12}}\left|(v_2 - v_3) \cdot \frac{x_2 - x_3}{|x_2 - x_3|}\right|H\left[(v_2 - v_3) \cdot (x_2 - x_3)\right]P^{(3,0)}\delta(|x_2-x_3|-2r_p)\\
\times\dd  x_2\dd  x_3\dd  v_2\dd  v_3\\
-\int_{\R^{12}}\left|(v_2^o - v_3^o) \cdot \frac{x_2 - x_3}{|x_2 - x_3|}\right|H\left[-(v_2^o - v_3^o) \cdot (x_2 - x_3)\right]P^{(3,0)}\delta(|x_2-x_3|-2r_p)\\
\times\dd  x_2\dd  x_3\dd  v_2\dd  v_3.
\end{multline*}
By \eqref{specular reflection}, the positive and negative parts cancel each other. For the same reason, we have
\begin{multline*}
    \int_{\R^{12}}w_1\cdot\frac{y_1-y_2}{|y_1-y_2|}P^{(1,2)}\delta(|y_1-y_2|-2r_g)\dd  y_1\dd  y_2\dd  w_1\dd  w_2\\
    =\frac{1}{2} \int_{\R^{12}}(w_1-w_2)\cdot\frac{y_1-y_2}{|y_1-y_2|}P^{(1,2)}\delta(|y_1-y_2|-2r_g)\dd  y_1\dd  y_2\dd  w_1\dd  w_2=0,
\end{multline*}
and
\begin{equation*}
    \int_{\R^{12}}(v_2-w_1)\cdot\frac{y_1-x_2}{|y_1-x_2|}P^{(2,1)}\delta(|y_1-x_2|-r_g-r_p)\dd  y_1\dd  x_2\dd  w_1\dd  v_2=0.
\end{equation*}

We decompose the remaining two integrals in \eqref{cpt:before using boundary conditions} into positive and negative parts, called the gain and loss terms, respectively, as done previously:
\begin{multline*}
    \p_tP^{(1,0)}+v_1\cdot\nabla_{x_1}P^{(1,0)}\\
    =(N_p-1)\int_{\R^{6}}\left|(v_1-v_2)\cdot\frac{x_1-x_2}{|x_1-x_2|}\right|\delta(|x_1-x_2|-2r_p)\bigg[P^{(2,0)}(x_1,x_2,v_1,v_2)\\
    \times H((v_1-v_2)\cdot(x_1-x_2))
    -P^{(2,0)}(x_1,x_2,v_1,v_2)H(-(v_1-v_2)\cdot(x_1-x_2))\bigg]\dd  x_2\dd  v_2\\
+N_g\int_{\R^{6}}\left|(v_1-w_1)\cdot\frac{x_1-y_1}{|x_1-y_1|}\right|\delta(|x_1-y_1|-r_g-r_p)\bigg[P^{(1,1)}(x_1,y_1,v_1,w_1)\\
\times H((x_1-y_1)\cdot (v_1-w_1))-P^{(1,1)}(x_1,y_1,v_1,w_1)H(-(x_1-y_1)\cdot (v_1-w_1))\bigg]\dd  y_1\dd  w_1.
\end{multline*}
We then apply the boundary condition \eqref{boundary condition for g} on the outgoing hemisphere:
\begin{multline*}
    \p_tP^{(1,0)}+v_1\cdot\nabla_{x_1}P^{(1,0)}\\
    =(N_p-1)\int_{\R^{6}}\left|(v_1-v_2)\cdot\frac{x_1-x_2}{|x_1-x_2|}\right|\delta(|x_1-x_2|-2r_p)\bigg[P^{(2,0)}(x_1,x_2,v_1^o,v_2^o)\\
    \times H((v_1-v_2)\cdot(x_1-x_2))
    -P^{(2,0)}(x_1,x_2,v_1,v_2)H(-(v_1-v_2)\cdot(x_1-x_2))\bigg]\dd  x_2\dd  v_2\\
+N_g\int_{\R^{6}}\left|(v_1-w_1)\cdot\frac{x_1-y_1}{|x_1-y_1|}\right|\delta(|x_1-y_1|-r_g-r_p)\bigg[P^{(1,1)}(x_1,y_1,v_1',w_1')\\
\times H((x_1-y_1)\cdot (v_1-w_1))-P^{(1,1)}(x_1,y_1,v_1,w_1)H(-(x_1-y_1)\cdot (v_1-w_1))\bigg]\dd  y_1\dd  w_1.
\end{multline*}
We then change variables from $x_2,y_1\in\R^3$ to the surface variable on the unit sphere $n\in\sph^2$ by setting $x_2=x_1-2r_pn$ and $y_1=x_1-(r_g+r_p)n$, respectively. We note that $\delta(|x_1-x_2|-2r_p)\dd  x_2=(2r_p)^2\dd  n$ and $\delta(|x_1-y_1|-r_g-r_p)\dd  y_1=(r_g+r_p)^2\dd  n$. We then apply the change $n\mapsto -n$ on the incoming hemisphere:
\begin{multline*}
    \p_tP^{(1,0)}+v_1\cdot\nabla_{x_1}P^{(1,0)}=(N_p-1)(2r_p)^2\int_{\R^{3}\times\sph^2}\left|(v_1-v_2)\cdot n\right|H((v_1-v_2)\cdot n)\\
    \times\bigg[P^{(2,0)}(x_1,x_1-2r_pn,v_1^o,v_2^o)-P^{(2,0)}(x_1,x_1+2r_pn,v_1,v_2)\bigg]\dd  n\dd  v_2\\
+N_g(r_g+r_p)^2\int_{\R^{3}\times \sph^2}|(v_1-w_1)\cdot n| H ((v_1-w_1)\cdot n)\bigg[P^{(1,1)}(x_1,x_1-(r_g+r_p)n,v_1',w_1')\\
-P^{(1,1)}(x_1,x_1+(r_g+r_p)n,v_1,w_1)\bigg]\dd  n\dd  w_1.
\end{multline*}

In fact, $P^{(s,l)}$ represents the $(s+l)$th marginal of the distribution function. We now establish its relationship with the density distribution functions for particles and gas molecules. The particle mass density distribution function $g^{(1,0)}$ is defined as  
\begin{multline*}
    g^{(1,0)}(t,x,v):=m_p\sum_{i=1}^{N_p}\int_{\R^{6(N_g+N_p-1)}} \tilde{P}^{(N_p,N_g)}(t,x_{1},\dots,x_{i-1},x,x_{i+1},\dots,x_{N_p}, Y_{(N_g)},\\
    v_1,\dots,v_{i-1},v,v_{i+1},\dots,v_{N_p},W_{(N_g)})\\
    \dd  X_{(i-1)}\dd  X^{(i+1)}_{(N_p)}\dd  Y_{(N_g)}\dd  V_{(i-1)}\dd  V^{(i+1)}_{(N_p)}\dd  W_{(N_g)}.
\end{multline*}

We shift the indices $1, 2, \dots, i-1$ forward by one position in the sequence. Specifically, we relabel $(x_1, x_2, \dots, x_{i-1}, v_1, v_2, \dots, v_{i-1})$ as $(x_2, x_3, \dots, x_i, v_2, v_3, \dots, v_i)$, and then substitute $(x, v)$ with $(x_1, v_1)$.

\begin{multline*}
    g^{(1,0)}(t,x_1,v_1)=m_p\sum_{i=1}^{N_p}\int_{\R^{6(N_g+N_p-1)}} \tilde{P}^{(N_p,N_g)}(t,x_{2},\dots,x_{i},x_1,x_{i+1},\dots,x_{N_p}, Y_{(N_g)}\\
    ,v_2,\dots,v_{i},v_1,v_{i+1},\dots,v_{N_p},W_{(N_g)})\dd  X^{(2)}_{(N_p)}\dd  Y_{(N_g)}\dd  V^{(2)}_{(N_p)}\dd  W_{(N_g)}.
\end{multline*}
For any permutation $\tilde{\sigma}_i\in S_{N_p-1}:\{2,\dots,N_p\}\to \{2,\dots,N_p\}$ with $1\le i\le N_p$, the following identity holds:
\begin{multline*}
    \int_{\R^{6(N_g+N_p-1)}} \tilde{P}^{(N_p,N_g)}(t,x_{2},\dots,x_{i},x_1,x_{i+1},\dots,x_{N_p}, Y_{(N_g)}\\
    ,v_2,\dots,v_{i},v_1,v_{i+1},\dots,v_{N_p},W_{(N_g)})\dd  X^{(2)}_{(N_p)}\dd  Y_{(N_g)}\dd  V^{(2)}_{(N_p)}\dd  W_{(N_g)}\\
    =\int_{\R^{6(N_g+N_p-1)}} \tilde{P}^{(N_p,N_g)}(t,x_{\tilde{\sigma}_i(2)},\dots,x_{\tilde{\sigma}_i(i)},x_1,x_{\tilde{\sigma}_i(i+1)},\dots,x_{\tilde{\sigma}_i(N_p)}, Y_{(N_g)}\\
    ,v_{\tilde{\sigma}_i(2)},\dots,v_{\tilde{\sigma}_i(i)},v_1,v_{\tilde{\sigma}_i(i+1)},\dots,v_{\tilde{\sigma}_i(N_p)},W_{(N_g)})\dd  X^{(2)}_{(N_p)}\dd  Y_{(N_g)}\dd  V^{(2)}_{(N_p)}\dd  W_{(N_g)}.
\end{multline*}
Summing over $\tilde{\sigma}_i\in S_{N_p-1}$, we rewrite $g^{(1,0)}(t,x_1,v_1)$ as
\begin{multline*}
    g^{(1,0)}(t,x_1,v_1)=m_p\sum_{i=1}^{N_p}\frac{1}{(N_p-1)!}\sum_{\tilde{\sigma}_i\in S_{N_p-1}}\int_{\R^{6(N_g+N_p-1)}}
    \tilde{P}^{(N_p,N_g)}\bigg(t,x_{\tilde{\sigma}_i(2)},\dots\\
    ,x_{\tilde{\sigma}_i(i)},x_1,x_{\tilde{\sigma}_i(i+1)},\dots,x_{\tilde{\sigma}_i(N_p)}, Y_{(N_g)},v_{\tilde{\sigma}_i(2)},\dots,v_{\tilde{\sigma}_i(i)},v_1,v_{\tilde{\sigma}_i(i+1)},\dots
    ,v_{\tilde{\sigma}_i(N_p)},W_{(N_g)}\bigg)\\
    \dd  X^{(2)}_{(N_p)}\dd  Y_{(N_g)}\dd  V^{(2)}_{(N_p)}\dd  W_{(N_g)}.
\end{multline*}
Define a permutation $\sigma_i:\{1,\dots,N_p\}\to\{1,\dots,N_p\}$ that maps $i$ to $1$ by
\[
\sigma_i(k):=\left\{\begin{aligned}
&\tilde{\sigma}_i(k+1), \quad& 1\le k \le i-1;\\
&1,\quad & k=i;\\
&\tilde{\sigma}_i(k),\quad &i+1\le k\le N_p.\\
\end{aligned}\right.
\]
Then we can combine the two summations:
\begin{multline*}
g^{(1,0)}(t,x_1,v_1)=m_pN_p\frac{1}{N_p!}\sum_{\sigma\in S_{N_p}}\int_{\R^{6(N_g+N_p-1)}} \tilde{P}^{(N_p,N_g)}\bigg(t,x_{\sigma(1)},\dots,x_{\sigma(N_p)}, Y_{(N_g)}\\
,v_{\sigma(1)},\dots,v_{\sigma(N_p)},W_{(N_g)}\bigg)
\dd  X^{(2)}_{(N_p)}\dd  Y_{(N_g)}\dd  V^{(2)}_{(N_p)}\dd  W_{(N_g)}.
\end{multline*}
The definitions of the averaged particle distribution function \eqref{definition of g bar} and the marginal \eqref{definition of marginal} show that
\begin{equation}\label{definition of F}
g^{(1,0)}(t,x,v)=m_pN_pP^{(1,0)}(t,x,v).
\end{equation}
For the same reason, we have
\begin{equation}\label{definition of f}
g^{(0,1)}(t,x,w)=m_gN_gP^{(0,1)}(t,x,w).
\end{equation}

Similarly, we can establish the connection between the second marginal $P^{(1,1)}$ and the two-particle distribution function $P^{(2,0)}$ using the same idea. The definition of $P^{(2,0)}$ is given by  
\begin{multline*}
    g^{(2,0)}(t,\xi_1,\xi_2,\nu_1,\nu_2)\\
    :=m_p^2\sum_{1\le i<j\le N_p}\int_{\R^{6(N_g+N_p-2)}} \tilde{P}^{(N_p,N_g)}(t,x_{1},\dots,x_{i-1},\xi_1,x_{i+1},\dots, x_{j-1},\xi_2,x_{j+1},\dots,x_{N_p}\\
    ,Y_{(N_g)},v_1,\dots,v_{i-1},\nu_1,v_{i+1}.\dots,v_{j-1},\nu_2,v_{j+1},\dots,v_{N_p},W_{(N_g)})\\
    \dd  X_{(i-1)}\dd  X^{(i+1)}_{(j-1)}\dd  X^{(j+1)}_{(N_p)}\dd  Y_{(N_g)}\dd  V_{(i-1)}\dd  V^{(i+1)}_{(j-1)}\dd  V^{(j+1)}_{(N_p)}\dd  W_{(N_g)}.
\end{multline*}
Using a change of variables, we can symmetrize it as follows:
\begin{multline*}
    g^{(2,0)}(t,\xi_1,\xi_2,\nu_1,\nu_2)\\
    :=m_p^2\sum_{1\le i<j\le N_p}\frac{1}{(N_p-2)!}\sum_{\tilde{\sigma}_{ij}\in S_{N_p-2}}\int_{\R^{6(N_g+N_p-2)}}
    \tilde{P}^{(N_p,N_g)}\bigg(t,x_{\tilde{\sigma}_{ij}(3)},\dots\\
    ,x_{\tilde{\sigma}_{ij}(i+1)},\xi_1,,x_{\tilde{\sigma}_{ij}(i+2)},\dots,x_{\tilde{\sigma}_{ij}(j)},\xi_2,x_{\tilde{\sigma}_{ij}(j+1)},\dots,x_{\tilde{\sigma}_{ij}(N_p)}, Y_{(N_g)}\\
    ,v_{\tilde{\sigma}_{ij}(3)},\dots,v_{\tilde{\sigma}_{ij}(i+1)},\nu_1,,v_{\tilde{\sigma}_{ij}(i+2)},\dots,v_{\tilde{\sigma}_{ij}(j)},\nu_2,v_{\tilde{\sigma}_{ij}(j+1)},\dots,v_{\tilde{\sigma}_{ij}(N_p)}\bigg)\\
    \dd  X_{(i-1)}\dd  X^{(i+1)}_{(j-1)}\dd  X^{(j+1)}_{(N_p)}\dd  Y_{(N_g)}\dd  V_{(i-1)}\dd  V^{(i+1)}_{(j-1)}\dd  V^{(j+1)}_{(N_p)}\dd  W_{(N_g)}.
\end{multline*}
We extend $\tilde{\sigma}_{ij}\in S_{N_p-2}$ to $\sigma_{ij}\in S_{N_p}$ by defining 
\[
\sigma_{ij}(k):=\left\{\begin{aligned}
&\tilde{\sigma}_{ij}(k+2), \quad& 1\le k \le i-1;\\
&1,\quad & k=i;\\
&\tilde{\sigma}_{ij}(k+1), \quad& i+1\le k \le j-1;\\
&2,\quad & k=j;\\
&\tilde{\sigma}_{ij}(k),\quad &j+1\le k\le N_p.\\
\end{aligned}\right.
\]
Sum over ordered pairs $i\neq j$ so that
\[
\frac{1}{(N_p-2)!}\sum_{i\neq j}\sum_{\tilde{\sigma_{ij}}}=N_p(N_p-1)\frac{1}{N_p!}\sum_{{\sigma\in S_{N_p}}}.
\]
We can hence define $g^{(2,0)}$ by
\[
g^{(2,0)}(t,x_1,x_2,v_1,v_2) = m_p^2 N_p(N_p-1) P^{(2,0)}(t,x_1,x_2,v_1,v_2).
\]
A similar computation yields  
\begin{align*}
&g^{(1,1)}(t,x_1,y_1,v_1,w_1) = m_p m_g N_p N_g P^{(1,1)}(t,x_1,y_1,v_1,w_1),\\
&g^{(0,2)}(t,y_1,y_2,w_1,w_2) = m_g^2 N_g(N_g-1) P^{(0,2)}(t,y_1,y_2,w_1,w_2).
\end{align*}
Therefore, we can write the equation for the particle density distribution function by replacing $P^{(i,j)}$
 with $g^{(i,j)}$:
 \begin{multline*}
    \p_tg^{(1,0)}+v_1\cdot\nabla_{x_1}g^{(1,0)}=\frac{(2r_p)^2}{m_p}\int_{\R^{3}\times\sph^2}\left|(v_1-v_2)\cdot n\right|H((v_1-v_2)\cdot n)\\
    \times\bigg[g^{(2,0)}(x_1,x_1-2r_pn,v_1^o,v_2^o)-g^{(2,0)}(x_1,x_1+2r_pn,v_1,v_2)\bigg]\dd  n\dd  v_2\\
+\frac{(r_g+r_p)^2}{m_g}\int_{\R^{3}\times \sph^2}|(v_1-w_1)\cdot n| H ((v_1-w_1)\cdot n)\bigg[g^{(1,1)}(x_1,x_1-(r_g+r_p)n,v_1',w_1')\\
-g^{(1,1)}(x_1,x_1+(r_g+r_p)n,v_1,w_1)\bigg]\dd  n\dd  w_1.
\end{multline*}

It is clear that the equation remains unclosed, even when combined with the equation for the gas molecules $g^{(0,1)}$. To achieve closure, we invoke the Enskog closure on ingoing half spheres, stated as follows:
\begin{equation}\label{def:correlation functions}
g^{(0,2)} = Y^{(0,2)}[f,f]f \otimes f, \quad g^{(2,0)} = Y^{(2,0)}[F,F] F \otimes F, \quad g^{(1,1)} = Y^{(1,1)}[F,f]F \otimes f,
\end{equation}
where, for notational simplicity, $F \equiv F(t,x,v) := g^{(1,0)}(t,x,v)$ and $f \equiv f(t,x,w) := g^{(0,1)}(t,x,w)$. The functionals $Y^{(i,j)}[\varphi,\phi]\equiv Y^{(i,j)}[\varphi,\phi](t,x)$, with $i,j\in\{0,1,2\}$, account for the correction to collision probabilities due to volume exclusion effects arising from finite particle size. Furthermore, $Y^{(1,1)}[F,f]$ is a symmetric functional with respect to $f$ and $F$; that is,
\[
Y^{(1,1)}[F,f] = Y^{(1,1)}[f,F].
\]

Applying a similar argument to the gas-molecule equation and integrating \eqref{liouville equation} with respect to $\dd  X_{(N_p)}\dd  V_{(N_p)}$ and $\dd  Y^{(2)}_{(N_g)}\dd  W^{(2)}_{(N_g)}$. 
The parametrization $x_1=y_1+(r_g+r_p)n$, $y_2=y_1-2r_gn$, the boundary condition on the outgoing half-sphere, the change of variable $n\mapsto-n$ on ingoing half-sphere, the multiplication by $m_gN_g$ together with $g^{(1,1)}=m_gm_pN_gN_pP^{(1,1)}$ produce the prefactor ${(r_p+r_g)^2}/{m_p}$. 
Therefore, we obtain the equation \eqref{coupled Enskog equations} with the collision integrals \eqref{def:collision integral E_0}, \eqref{def:collision integral E_1}, \eqref{def:collision integral E_2}, and \eqref{def:collision integral E_3}.



\section{Conclusion}
We have verified the coupled Enskog equations from the underlying particle system at a formal level. Two essential features emerge that are absent in the single-species Boltzmann case:
\begin{itemize}
    \item the delocalization effects between different species;
    \item the precise coefficients appearing in front of each collision integral.
\end{itemize}
The directions of these delocalization effects are important for understanding their effects on macroscopic quantities such as pressure, density, and bulk velocity. Meanwhile, the coefficients in front of the collision integrals play a role when we nondimensionalize the coupled Enskog equations (for example, see \cite{CCG2026}).

A natural question arises: can a general collision kernel be used instead of the hard-sphere case for the coupled Enskog equations? The answer is yes, but the generalization becomes meaningless in a physical sense. For general interaction potentials, the molecular radius is irrelevant; what matters is the range of the force. Thus, the system is better viewed as a short-range potential model. For dilute gases, this has already been studied in \cite{PSS2014}. For dense gases, however, substantial difficulties appear: the collision frequency may fail to remain finite, recollisions may no longer be negligible, and key estimates (such as estimate (5.5) in \cite{PSS2014}) break down, since $\varepsilon^2 (N-j)$ is no longer less than $1$.

The rigorous derivation of such models is certainly of great interest. Yet even the rigorous derivation of the Enskog equation for single species remains an open problem at the time of this writing.

\section*{Acknowledgments}
This work forms part of the author's doctoral thesis \cite{thesisChen} at Sorbonne Université.
\appendix
\section{Proof of Proposition \ref{prop:entropy inequality for BB system}}\label{apd}
We take the derivative of the total entropy \eqref{def:entropy for BB system} of the coupled Boltzmann equations 
\begin{multline*}
    \dt S_{\mathrm{BB}}=\int_{\R^3\times \R^3}\p_tf(t,x,w)\ln f(t,x,w)\dd  x\dd  w+\dt M_g\\
    +\int_{\R^3\times \R^3}\p_tF(t,x,v)\ln F(t,x,v)\dd  x\dd  v+\dt M_\ell.
\end{multline*}
We have already shown that both species conserve their mass. Since $f$ and $F$ satisfy the equations in \eqref{coupled boltzmann equations}, it follows that
\begin{multline*}
    \dt S_{\mathrm{BB}}=\int_{\R^3\times \R^3}(-w\cdot\nabla_xf+\mathcal{B}[f]+\mathcal{B}_1[f,F])\ln f\dd  x\dd  w\\
    +\int_{\R^3\times \R^3}(-v\cdot\nabla_xF+\mathcal{B}[F]+\mathcal{B}_2[F,f])\ln F\dd  x\dd  v.
\end{multline*}
Noting that
\[
w\cdot \nabla_x f\ln f=\div_x(wf \ln f)-\div_x(wf),
\]
and the classical H-theorem for the Boltzmann collision integral gives
\[
\int_{\R^3}\mathcal{B}[f]\ln f\dd  w\le 0,
\]
we obtain
\[
\dt S_{\mathrm{BB}}\le \int_{\R^3\times \R^3}\mathcal{B}_1[f,F]\ln f\dd  x\dd  w+\int_{\R^3\times \R^3}\mathcal{B}_2[F,f]\ln F\dd  x\dd  v.
\]
Performing the pre-to-post transformation on the gain terms of the collision integrals $\mathcal{B}_1$ and $\mathcal{B}_2$, we find
\begin{multline*}
    \int_{\R^3}\mathcal{B}_1[f,F]\ln f\dd  w+\int_{\R^3}\mathcal{B}_2[F,f]\ln F\dd  v
    \\
    =\int_{\R^3\times\R^3\times \sph^2} f(w)F(v)\ln\frac{f(w')F(v')}{f(w)F(v)}b(v-w,n)\dd  n\dd  v\dd  w.
\end{multline*}
Thanks to \eqref{fundamental inequality},
we deduce
\[
\dt S_{\mathrm{BB}}\le \int_{\R^3\times\R^3\times \sph^2\times \R^3} [f(w')F(v')-f(w)F(v)]b(v-w,n)\dd  n\dd  v\dd  w\dd  x=0,
\]
where the last equality follows from the change of variables $(v,w)\mapsto (v',w')$ given by \eqref{collision formula between different species} and the invariance of the collision kernel $b$.


\bibliographystyle{alpha}
\bibliography{ref_PSEE}
\end{document}